\documentclass[sn-mathphys-num]{sn-jnl}
\usepackage{mathtools}
\usepackage[numbers]{natbib}
\usepackage{graphicx}
\usepackage{multirow}
\usepackage{amsmath,amssymb,amsfonts, amsthm,  mathrsfs}
\usepackage[title]{appendix}
\usepackage{xcolor}
\usepackage{textcomp}
\usepackage{manyfoot, booktabs}
\usepackage{algorithm, algorithmicx,  algpseudocode,  listings}
\usepackage[normalem]{ulem}

\usepackage{comment, soul}

\usepackage{bm}                 
\DeclareMathOperator*{\argmin}{\textrm{argmin}\,} 
\DeclareMathOperator*{\argmax}{\textrm{argmax}\,} 
\def\dom             {\textrm{dom}} 	 
\def\grad            {\textrm{grad}}     
\def\proj            {\textrm{proj}}     
\def\st              {\,\textrm{s.t.}\,} 
\def\tr              {\textrm{Tr}}       
\def\zeros           {\bm{0}}            
\def\IN {\mathbb{N}} 		
\def\IR  {\mathbb{R}} 	    
\def\IRm {\IR^{m}} 			
\def\IRn {\IR^{n}} 			
\def\IRr {\IR^{r}} 			
\def\IRmn{\IR^{m \times n}} 
\def\IRmr{\IR^{m \times r}} 
\def\IRrn{\IR^{r \times n}} 

\def\IS {\mathbb{S}} 		%
\def\A {\bm{A}}
\def\B {\bm{B}}

\def\D {\bm{D}}

\def\H {\bm{H}} 
\def\I {\bm{I}}

\def\M {\bm{M}}

\def\P {\bm{P}}

\def\S {\bm{S}}

\def\W {\bm{W}}

\def\Z {\bm{Z}}

\def\a {\bm{a}}
\def\b {\bm{b}}
\def\c {\bm{c}}

\def\e {\bm{e}}

\def\h {\bm{h}}

\def\m {\bm{m}}

\def\p {\bm{p}}
\def\q {\bm{q}}
\def\r {\bm{r}}

\def\t {\bm{t}}
\def\u {\bm{u}}
\def\v {\bm{v}}
\def\w {\bm{w}}
\def\x {\bm{x}} 
\def\y {\bm{y}}
\def\z {\bm{z}}

\def\cE { \mathcal{E}}

\def\cM { \mathcal{M}}
\def\cN { \mathcal{N}}
\def\cO { \mathcal{O}}
\def\cP { \mathcal{P}}

\def\cR { \mathcal{R}}

\def\cU { \mathcal{U}}

\def\cX { \mathcal{X}}

\def\balpha   {\boldsymbol{\alpha}}

\def\bzeta    {\boldsymbol{\zeta}}

\def\bxi      {\boldsymbol{\xi}}

\newtheorem{theorem}{Theorem}
\newtheorem{proposition}[theorem]{Proposition}

\newtheorem{remark}{Remark}
\newtheorem{lemma}[theorem]{Lemma}
\newtheorem{definition}{Definition}

\begin{document}

\title[Chordal-NMF with Riemannian MU]{A new Riemannian Multiplicative Updates for Chordal Nonnegative Matrix Factorization}

\author*[1]{\fnm{Flavia} \sur{Esposito}}\email{flavia.esposito@uniba.it}
\equalcont{These authors contributed equally to this work.}

\author[2]{\fnm{Andersen} \sur{Ang}}
\equalcont{These authors contributed equally to this work.}

\affil*[1]{\orgdiv{Department of Mathematics}, \orgname{Università degli Studi di Bari Aldo Moro}, \orgaddress{\city{Bari}, \country{Italy}}}

\affil[2]{\orgdiv{School of Electronics and Computer Science}, \orgname{University of Southampton}, \orgaddress{\city{Southampton}, \country{United Kingdom}}}

\abstract{Nonnegative Matrix Factorization (NMF) is the problem of approximating a given nonnegative matrix $\M$ through the product of two nonnegative low-rank matrices $\W$ and $\H$. 
Traditionally NMF is tackled by optimizing a specific objective function evaluating the quality of the approximation. 
This assessment is often done based on the Frobenius norm (F-norm).
In this work, we argue that the F-norm, as the ``point-to-point'' distance, may not always be appropriate. 
Viewing from the perspective of cone, NMF may not naturally align with F-norm.
So, a ray-to-ray chordal distance is proposed as an alternative way of measuring the 
quality of the approximation.
As this measure corresponds to the Euclidean distance on the sphere, it motivates the use of  manifold optimization techniques.
We apply Riemannian optimization technique to solve chordal-NMF by casting it on a manifold.
Unlike works on Riemannian optimization that require the manifold to be smooth, the nonnegativity in chordal-NMF defines a non-differentiable manifold.
We propose a Riemannian Multiplicative Update (RMU), 
and showcase the effectiveness of the chordal-NMF on synthetic and real-world datasets.}
\maketitle

\section{Introduction}\label{sec:intro}
Given a nonnegative matrix $\M \in \IRmn_+$ and a  rank $r \leq \min\{m,n\}$,
the goal of Nonnegative Matrix Factorization (NMF) is to find factor matrices $\W \in \IRmr_+$ and $\H \in \IRrn_+$ such that $\M \approx \W\H$~\cite{gillis2020nonnegative}. 
NMF is commonly achieved by minimizing the Frobenius norm (F-norm) $\| \M - \W \H \|_{\textrm{F}} = \sqrt{\sum_{ij} (M_{ij} - (WH)_{ij})^2}$ (where $M_{ij}$ is the $(i,j)$th-entry of $\M$) that measures the quality of the approximation.

For $\M\approx\W\H$ with nonnegativity constraints, the point cloud $\M$ is contained within a polyhedral cone generated by the $r$ columns of $\W$ with nonnegative weights encoded in $\H$, see Fig.~\ref{fig:cone}.
This conic view of NMF suggests that the point-to-point distance $M_{ij} - (\W\H)_{ij}$ in the F-norm does not naturally fit NMF.
In this work, we propose to replace F-norm by a ray-to-ray distance that we call \textit{chordal distance}, detailed in the next section.
We are interested in solving
\begin{equation}\label{def:chordal_NMF}
\begin{array}{l}
    \displaystyle
\hspace{-0.3cm}
\argmin\limits_{\W\geq \zeros,\H\geq \zeros}
\Bigg\{
F(\W,\H)
\coloneqq
\dfrac{1}{n}
\displaystyle \sum\limits_{j=1}^n
\bigg(
1-
\dfrac{\langle \m_{:j} , \W\h_{:j} \rangle }{\| \m_{:j} \|_2 \|  \W\h_{:j} \|_2}
\bigg)
\Bigg\},
\end{array}
\tag{Chordal-NMF}
\end{equation}
where the objective function $F:\IRmr\times\IRrn\to\IR$ is defined as the chordal distance between $\W\H$ and $\M$, 
$\| \cdot \|_2$ is the Euclidean norm, 
$\langle \cdot, \cdot\rangle$ is the Euclidean inner product, 
and $\m_{:j}$ is the $j$th column of $\M$.
The nonnegativity constraints $\geq \zeros$ is element-wise, where $\zeros$ is zero matrix of the appropriate size.
For simplicity, we assume $\M \in \IRmn_+$ has no zero column, and $\M$ is normalized  that $\| \m_{:j}\|_2=1$ in~\eqref{def:chordal_NMF}.
We give the motivation of using ray-to-ray distance in \textsection\ref{sec:chordal_distance}.
\begin{figure}[!ht]\centering
    \includegraphics[width=0.75\textwidth]{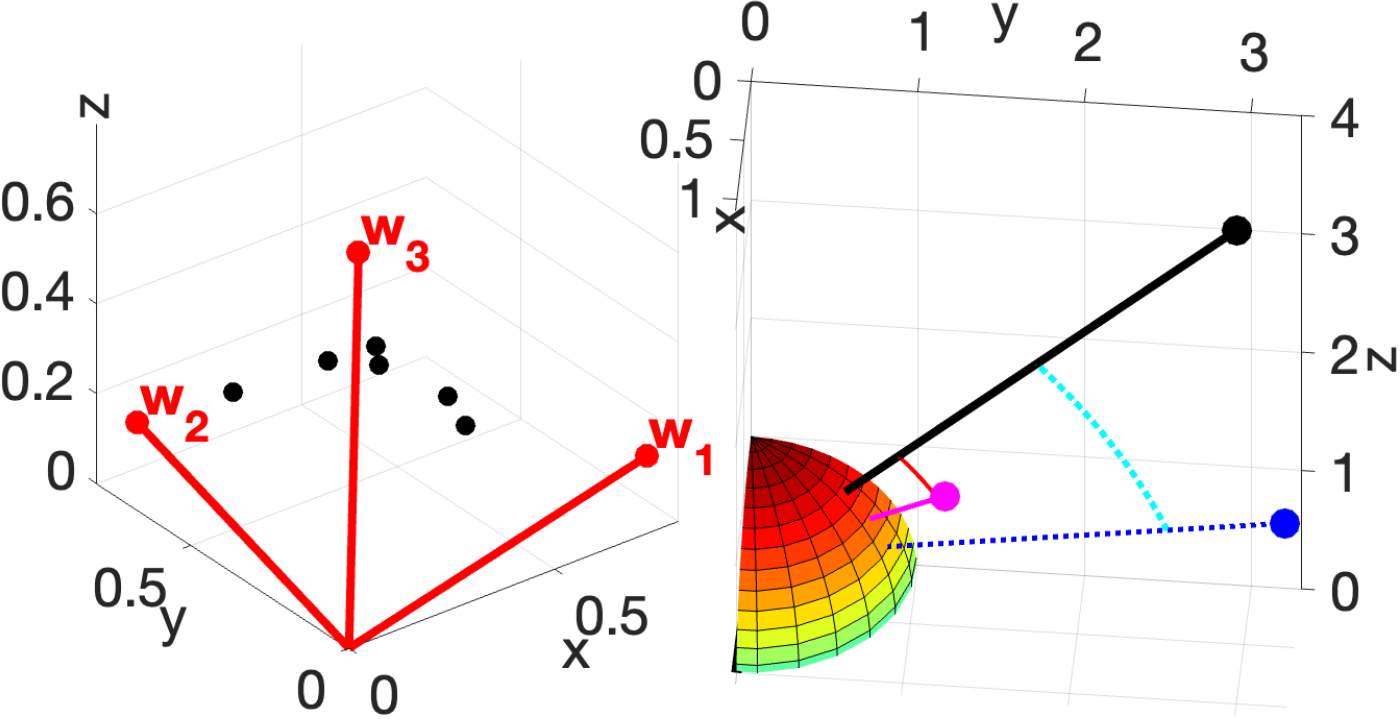}
    \caption{Left: Picture of a rank-3 NMF.
    Data points (in black) that represent the columns of $\M$ encapsulated by a polyhedral cone generated by $r=3$ columns $\w_1, \w_2, \w_3$ (in red).
    Right: The sphere $\IS^{2}_+$ in the nonnegative orthant $\IR^3_+$.
    The black dot is closer to the blue dot in Euclidean distance, but closer to the pink dot in chordal distance, which is equivalent to the geodesic arc length on the sphere.
    }
    \label{fig:cone}
\end{figure}

\noindent\textbf{Contribution.}
Our contributions are 3-folds.
\begin{enumerate}
\item We propose a new model \eqref{def:chordal_NMF}.
To the best of our knowledge, such problem is new and has not been studied in the NMF literature.

\item Solving Chordal-NMF is nontrivial, it is a nonsmooth nonconvex and block-nonconvex problem.
We propose a Block Coordinate Descent (BCD) algorithm with Riemannian Multiplicative Update (RMU) for solving Chordal-NMF.

\begin{itemize}
    \item We derive the Riemannian gradient of the objective function $F$.
    \item The nonnegativity constraints in Chordal-NMF introduce nondifferentiability in the manifold and make some manifold techniques ineffective.
    We propose RMU to solve the nonsmoothness issue in the optimization.
    In particular, we show that, if the initial variable in the algorithm is feasible, the whole sequence is guaranteed to be feasible.
\end{itemize}

\item We showcase the effectiveness of the Chordal-NMF in \textsection\ref{sec:Exp}.

\end{enumerate}

\noindent\textbf{Paper organization.}
In the remaining of this section, we describe the overall algorithmic framework on solving Chordal-NMF, as well as the motivation of the chordal distance.
In \textsection\ref{sec:tool} we review Riemannian optimization techniques.
In \textsection\ref{sec:H} and \textsection\ref{sec:W} we discuss how we update the variable $\H$ and $\W$, respectively.
\textsection\ref{sec:Exp} contains the experiments and \textsection\ref{sec:conc} sketches the conclusion.
\\

\noindent\textbf{Notation.}
We use $\{$italic, bold italic, bold italic capital$\}$ letters to denote $\{$scalar, vector, matrix$\}$ resp..
For a matrix $\A$, we denote $\a_{:j}$ its $j$th column, $\a_{j:}$ its $j$th row, and $\A^\top$ its transpose.
We denote $(\M)_{ij}$ or $M_{ij}$ the $ij$ component of $\M$.
The notation $\langle \bxi, \bzeta \rangle$ denotes the Euclidean inner product in the standard basis,
and $\| \bxi \|_2$ denotes the Euclidean norm of $\bxi$.
For a vector $\v \neq \zeros$, we denote $\hat{\v}=\v / \| \v \|_2$ the unit vector of $\v$.
If $\v = \zeros$ we define $\hat{\v}$ to be any unit vector.
We use $[\theta]_+ \coloneqq \max\{0,\theta\}$ element-wise.
\\

\noindent\textbf{Useful tools.}
We list two tools.
The first one is useful for deriving the Euclidean gradient of the objective function.

\begin{proposition}\label{prop:grad_f}
Let $f(\x) = \langle \A \x + \b, \c \rangle / \| \D \x + \e \|_2$ with $\A \in \IRmn$, $\b \in \IRm$, $\c \in \IRm, \D \in \IR^{p \times n}, \e \in \IR^p$ and $\D \x + \e \neq \zeros$, the Euclidean gradient $\nabla f$, is 
\[
\nabla f(\x)
=
\dfrac{\| \D \x + \e \|_2^2 \A^\top \c
-
\langle \A \x + \b, \c \rangle
\D^\top (\D \x + \e)}
{\| \D \x + \e \|_2^3 }
.
\]
We put the proof in the appendix.
\end{proposition}

We make use of tensor product in this work.
Let $V$ and $W$ be two vector spaces with inner product.
Let $\v \otimes \w \in V \otimes W$ be the tensor product of $\v \in V$, $\w \in W$, then 
\begin{equation}\label{formula:tensor}
\langle \v, \x \rangle \w = (\w \otimes \v) \x.
\end{equation}

\noindent\textbf{Block Coordinate Descent (BCD).}
We use BCD to solve the Chordal-NMF, see Algorithm~\ref{alg:BCD}, where we also hide the factor $1/n$.

\begin{algorithm}
\caption{BCD for Chordal-NMF, with a starting point $(\W_0, \H_0)$}
\label{alg:BCD}
\begin{algorithmic}[1]
\For{$k=1,2,...$}
\For{$j=1,2,...,n$ (update $\displaystyle \H_{k+1} $ column-wise)} 
        \State \fbox{h-subproblem} 
        $ \displaystyle \h_{:j, k+1} = \argmax_{\h\geq \zeros}  
        \dfrac{\langle \m_{:j} , \W\h \big\rangle }{\|  \W\h \|_2}.$
\EndFor
\State \fbox{W-subproblem}
$\displaystyle 
\W_{k+1}
=
\argmax_{\W\geq \zeros} \hspace{-1mm} 
\Bigg\{
F(\W;\H)
=
\displaystyle
\sum_{j=1}^n
\dfrac{\langle \m_{:j} , \W\h_{:j} \rangle }{\|  \W\h_{:j} \|_2}
\Bigg\}
\Bigg|_{\H = \H_{k+1}}.$
\EndFor
\end{algorithmic}
\end{algorithm}

\noindent\textbf{Chordal-NMF is asymmetric.}
Usually considering that NMF in F-norm is symmetric: $\| \M - \W \H \|_F = \| \M^\top - \H^\top \W^\top\|_F$, we can use the same procedure updating $\H$ to update $\W$ (with a transpose).
Chordal-NMF is asymmetric: it measures the cosine distance between $\m_{:j}$ and $\W\h_{:j}$, not between the rows $\m_{i:}$ and $\w_{i,:}\H$.
This leads to the W-subproblem and H-subproblem have different structure, and we use different approaches to solve the subproblems.
We solve H-subproblem column-wise as in line 3 in Algorithm~\ref{alg:BCD}, to be discussed in \textsection\ref{sec:H}.
We solve W-subproblem matrix-wise, to be discussed in \textsection\ref{sec:W}.

\subsection{Deriving and motivating chordal distance}\label{sec:chordal_distance}
From the conic view of NMF (see introduction), the geometry of chordal distance can be seen as the Euclidean distance between unit vectors on the sphere.
Writing $\| \M - \W \H \|_F^2 = \sum_j \| \m_{:j} - \W \h_{:j} \|_2^2$, 
if both $\m_{:j}$ and $\W \h_{:j}$ have unit $\ell_2$-norm, then by the Euclidean inner product $\langle \bxi, \bzeta \rangle = \| \bxi \|_2 \| \bzeta \|_2 \cos \theta(\bxi, \bzeta)$ with $\theta(\bxi, \bzeta)$ the angle between the vectors $\bxi$, $\bzeta$, we have
\begin{equation}\label{eqn:NMF_Fnorm_expand_unit}
\dfrac{1}{2} \| \M - \W \H \|_F^2
~=~
\displaystyle
\sum\limits_{j=1}^n 
\Big( 1 - \big\langle \m_{:j}, \W \h_{:j} \big\rangle 
\Big)
~=~
\displaystyle
\sum\limits_{j=1}^n 
\Big( 1 - \cos \theta(\m_{:j}, \W \h_{:j})
\Big)
.
\end{equation}
While the normalization of $\m_{:j}$ is assumed, the normalization of $\W \h_{:j}$ is important here.
Note that enforcing $\W \h_{:j}$ to have unit $\ell_2-$norm forms a nonconvex constraint.

The expression \eqref{eqn:NMF_Fnorm_expand_unit} tells that the normalized Euclidean distance is the cosine distance.
Cosine distance is insensitive to the length of vector, see Fig.\ref{fig:cone} for an example.
\begin{itemize}
\item \textbf{Quotient space interpretation.}
From \eqref{eqn:NMF_Fnorm_expand_unit}, we define a new objective in \eqref{def:chordal_NMF} for measuring the distance between $\m_{:j}$ and $\W\h_{:j}$ by angle.
The division of the norm $\| \W \h_{:j}\|_2$ in \eqref{def:chordal_NMF} collapses all vectors  $\x$ that are proportional to $\W\h$ onto a single point on the unit sphere. 
In this way, the chordal distance becomes purely angle-based, since the length information is ignored.

\item \textbf{Haversine interpretation.}
Navy navigation~\cite{robusto1957cosine} makes use of the haversine $\textrm{hav}(\theta) \coloneqq (1-\cos \theta)/2$, where $\theta = d/r$ is the central angle, defined as the distance $d$ between two points along a great circle of the sphere, normalized by the radius of the sphere $r$.
The expression $1-\cos\theta$ in \eqref{eqn:NMF_Fnorm_expand_unit} is $\textrm{hav}(\theta)$ with $r=1$.

\item \textbf{Sphere interpretation.}
In $\IRm$ with Euclidean inner product $\langle \u, \v \rangle$ and norm $\| \u \|_2 = \sqrt{\langle \u, \u \rangle}$, the map $\bxi \mapsto \bxi / \| \bxi \|_2$ in \eqref{def:chordal_NMF}
sends nonzero vector $\bxi \in \IRm$ to the (nonconvex) unit sphere 
$
\IS^{m-1} 
\coloneqq\{
\bxi \in \IRm | \| \bxi \|_2 = \sqrt{\langle \bxi, \bxi \rangle}  = 1
\}
$.
Considering $\u \neq \zeros, \v \neq \zeros$, the following function
\[\bar{f}_{\text{chord}} :
\IRm \times \IRm \rightarrow  [0, 2] \cup \{\infty\} \cup  \{\pm \sqrt{-1}\infty\} \quad\text{as}\quad
\u,\v \mapsto 
\sqrt{
2 - 2   \dfrac{\langle \u , \v \rangle}{\| \u \|_2 \| \v \|_2} 
},
\]
can be seen as the Euclidean distance between two unit vectors $\hat{\u}$, $\hat{\v}$ on $\IS^{m-1}$, since
\[
\bar{f}_{\text{chord}}(\u,\v) 
=
\sqrt{
\Big\langle
\dfrac{\u}{\|\u\|_2}-\dfrac{\v}{\|\v\|_2}, \,
\dfrac{\u}{\|\u\|_2}-\dfrac{\v}{\|\v\|_2} 
\Big\rangle
}
=
\left\|\frac{\u}{\|\u\|_2}-\frac{\v}{\|\v\|_2}\right\|_2
=
\|\hat{\u}  - \hat{\v} \|_2.
\]
For this function it should be noted that:
\begin{itemize}
\item $\bar{f}_{\text{chord}}$ is undefined at $\u = \zeros$ and/or $\v = \zeros$, it take $\infty$ or complex values $\pm \sqrt{-1}\infty$.

\item $\bar{f}_{\text{chord}}$ is nondifferentiable with respect to (wrt.) parallel unit vectors as
$\bar{f}_{\text{chord}}(\hat{\u},\hat{\v}) 
\coloneqq \sqrt{2 - 2}
= |0|
$. 

\item In the Euclidean case, $\bar{f}_{\text{chord}}$ is non-convex wrt. $\u$, $\v$ and the pair $(\u,\v)$.
\end{itemize}
These undesirable properties lead to choice of chordal distance below.
\begin{definition}\label{def:chord_dist}
On $\IS^{m-1}$ with inner product $\langle \u, \v \rangle$,
we define
\[
f_{\text{sq-chord}} : \IS^{m-1} \times \IS^{m-1} \rightarrow [0, 4] 
:
\u,\v \mapsto 2 - 2 \langle \u , \v \rangle.
\]
\end{definition}
\end{itemize}

Given the above discussion and the fact that the chordal distance is non-Euclidean, we employ Riemannian optimization methods in this work.

\section{Background of optimization on manifold}\label{sec:tool}
In this section, we discuss about nonnegative-constrained optimization, then we review manifold optimization. 
Lastly, we present the Riemannian Multiplicative Update with some analysis.

Let $E$ be a linear space (e.g., $\IRn,\IRmr$) with an inner product $\langle\cdot,\cdot\rangle_E$ and an induced norm $\|\cdot\|_E$.
Given a differentiable function $f : E \to \IR$ and the nonnegative orthant $\cX \coloneqq \{ \x \in E \,|\, \x \geq \zeros\}$, consider the minimization problems
\[
\cP_0 : \argmin_{\x} f(\x) ~\st~ \x \in \cM, \, \x \in \cX
~~\overset{\text{restriction}}{\iff}~~
\cP_1 : \argmin_{\x \in \cM} f(\x) ~\st~ \x \in \cX,
\]
with the constraint set $\cM \coloneqq \{ \x \in E \,|\, h(\x) = 0 \}$ under the defining function $h(\x)$.
We focus on convex compact set $\cM$ being a smooth embedded submanifold of $\IRn$, where $h(\x)$ is many-times continuously differentiable.

In Euclidean optimization, $\x \in \cM $ is a constraint in $(\cP_0)$.
We ``remove'' such constraint using a restricted function $f\big|_\cM : \cM \to \IR$.
Here $f$ is the \textit{extension} of $f\big|_\cM$ that extends $\dom f\big|_\cM$ from $\cM$ to $E$.
In this way, $(\cP_0)$ can be written as a constrained manifold optimization $(\cP_1)$, which can further be converted into a unconstrained problem using indicator $\iota(\x) : E \to \IR \cup \{+\infty\}$  that $\iota(\x) = 0$ if $\x \in \cX$ and $\iota(\x) = +\infty$ if $\x \notin \cX$.
The problems are related, solving one of them will help solving the others.
Below we review approaches in manifold optimization on solving these problems.

\subsection{Nonnegative-constrained manifold optimization}\label{sec:tool:subsec:review}
Solving $(\cP_0)$ directly by manifold techniques is nontrivial.
Manifold optimization refers to an optimization on a smooth (differentiable) manifold.
The set $\cX$ is nonsmooth, same for $\cM_+ \coloneqq \cM \cap \cX$\footnote{For any point $\x$ in the boundary of $\cM_+$, there does not exist a $\IRn$-homeomorphic neighborhood in $\cM$.} (assumed nonempty).
Thus, several manifold optimization techniques do not have convergence guarantee for solving $(\cP_0)$.
The indicator $\iota$ is a nonsmooth convex lower semicontinuous (l.s.c) function~\cite{hosseini2017characterization}, making $(\cP_1)$ written in indicator form not satisfying the assumptions in some existing works in nonsmooth Riemannian optimization, e.g. Projected Gradient Descent on Riemannian Manifolds~\cite{hauswirth2016projected},
Riemannian proximal gradient~\cite{huang2022riemannian},
and Manifold proximal gradient~\cite{chen2020proximal}.
\\

\noindent\textbf{Dual methods: not strictly feasible.}
We can also solve $(\cP_0)$ by Riemannian Augmented Lagrangian multiplier (RALM)~\cite{liu2020simple} 
or solve $(\cP_1)$ using Riemannian Alternating Direction Method of Multipliers (RADMM)~\cite{kovnatsky2016madmm}.
The sequence generated by these methods is not strictly-primal-feasible and it is dangerous to use an infeasible variable to update the other block of variables in the BCD framework.
Therefore they are not feasible for our application.
\\

\noindent\textbf{Euclidean projected gradient is infeasible.}
We can solve $(\cP_0)$ by Euclidean projected gradient descent (EPG).
The projection $\proj_{\cM_+ }(  \z ) = {\displaystyle \argmin_{\x \in \cM_+}} \frac{1}{2}\| \x - \z \|_E^2$ takes an alternating projection method such as the Dykstra's algorithm \cite{boyle1986method} to solve, which can be expensive for our purpose.
\\

\noindent\textbf{Manifold projection-free method is expensive.}
Projection-free method, such as the Riemannian Frank-Wolfe (RFW)~\cite{weber2023riemannian}, has an expensive subproblem that requires the computation of exponential map and geodesic.
\\

\noindent\textbf{Our solution: Riemannian multiplicative update (RMU).}
In this work we provide a projection-free method that only requires the computation of the Riemannian gradient part $\grad f$.
To tackle these technical issues, we propose a cost-effective Riemannian multiplicative update (RMU) for solving Problem $(\cP_1)$.
Our approach is motivated by the research of NMF~\cite[\textsection8.2]{gillis2020nonnegative}.
The advantages of RMU are
\begin{itemize}
    \item The expensive projection $\proj_{\cM_+}$ is not required.
    \item Unlike the dual approaches, if the initial variable in the algorithm is feasible, the whole sequence is guaranteed to be feasible, see Proposition~\ref{prop:RMU_update}.
\end{itemize}
We introduce RMU in \textsection\ref{sec:tool:subsec:RMU}, before that we review  Riemannian optimization below.

\subsection{Background of Riemannian optimization}\label{sec:tools:subsec:background}
Riemannian optimization, or manifold optimization, has a long history~\cite{luenberger1972gradient,gabay1982minimizing,edelman1998geometry,absil2008optimization,udriste2013convex,rapcsak2013smooth,bento2017iteration,boumal2023introduction}.
So we give the minimum material on manifold optimization for the paper.
\\

\noindent\textbf{Submanifold and ambient space.}
Let $E$ be an Euclidean space with $dim(E)>k$, with inner product $\langle\cdot,\cdot\rangle$, and induced norm $\|\cdot\|$. 
Considering the smooth function $h : E \to \IR^k$, with $\textrm{D}h(\x)[\v] = \langle \grad h(\x), \v\rangle_E$ its differential with full
rank $k$ for all $\x$ such that $h(\x) = \zeros_k$, the set $\cM \coloneqq \{ \x \in E ~|~ h(\x) = \zeros_k \}$ is an embedded submanifold of $E$ of dimension $\dim E -k$ (see \cite[Ch7.7]{boumal2023introduction}).
In this work we focus on $h$ in the form $h(\x)=\langle\x,\x\rangle_E-1$.
For a manifold $\cM$ in $E$, we call $E$ the ambient space of $\cM$.
In this work, all ambient spaces are some specific linear vector space $E$, such as $\IRr$ and $\IRmr$.
\begin{itemize}
\item In the h-subproblem, $\cM$ is the ``shell'' of an ellipsoid, it has 1 dimension  lower than its ambient Euclidean space $\IRr$. 

\item In the W-subproblem, $\cM$ is the ``shell'' of a twisted spectrahedron.
\end{itemize}

\noindent\textbf{Tangent space and projection.}
Let $\textrm{ker}$ be the kernel of a matrix, the tangent space of $\cM$ at a reference point $\x$, denoted as $T_{\x}\cM$, is defined as the kernel of $\textrm{D}h(\x)$, i.e., 
$
T_{\x}\cM 
\coloneqq
\text{ker}\textrm{D}h(\x) 
=
\{ \v \in E ~|~ \textrm{D}h(\x)[\v]  = 0\}
$.
$T_{\x}\cM$ collects vectors $\v \in E$ that is tangent to $\cM$ at $\x$.
The orthogonal projection to $T_{\x}\cM$ is the linear map $\proj_{T_{\x}\cM}: \cM \to T_{\x}\cM$ characterized by the following properties \cite[Ch3, Def3.60]{boumal2023introduction}:
\begin{itemize}
    \item Range: $\textrm{Im}(\proj_{T_{\x}\cM})=T_{\x}\cM$;
    \item Projection: $\proj_{T_{\x}\cM}\circ \proj_{T_{\x}\cM}= \proj_{T_{\x}\cM}$;
\item Orthogonality: $\langle\u-\proj_{T_{\x}\cM}(\u),\y\rangle=0$ for all $\y\in T_{\x}\cM$ and $\u\in\cM$.
\end{itemize}
In this work, we have manifolds defined by smooth functions $h$ \cite[\textsection7.7]{boumal2023introduction}, we  use the
orthogonal projection $\proj_{T_{\x}\cM}: E \to T_{\x}\cM$, based on orthogonal decomposition of a vector space as
$\v = \proj_{T_{\x}\cM}(\v) +   \textrm{D}h(\x)^*[\alpha]$,
where $\textrm{D}h(\x)^*[\alpha]$ is the adjoint of $\textrm{D}h(\x)[\v]$, and $\alpha \in \IR$ plays the role of the dual variable (technically called covector) as the
solution of 
$ \argmin_{\alpha \in \IR} \| \v - \textrm{D}h(\x)^*[\alpha] \|_E^2 = (\textrm{D}h(\x)^*)^\dagger[\v]$, 
where $\dagger$ is pseudo-inverse.
This gives $\proj_{T_{\x}\cM}(\v) = \v - \textrm{D}h(\x)^* \big[
(\textrm{D}h(\x)^*)^\dagger[\v] \big]$.
Note that if the tangent space $\proj_{T_{\x}\cM}$ equals the ambient space, then the projection is not necessary.

\begin{remark}[Notational difference]
In \cite[Ch3, Def3.60]{boumal2023introduction} $\proj_x$ denotes orthogonal projection, here for indicating orthogonal projection from $E$ to $T_{\x}\cM$, we use $\proj_{T_{\x}\cM}$.
\\
\end{remark}

\noindent\textbf{Retraction.}
A point $\x$ that is originally sitting on a manifold $\cM$ may go outside $\cM$ after a gradient operation, so we need to pull it back onto $\cM$. 
This can be achieved by a particular smooth map, known as the retraction $\cR: T_{\x}\cM \to \cM$ with $\v\mapsto \cR_{\x}(\v)$.
A retraction maps a point on the tangent space $T_{\x}\cM$ onto $\cM$ such that each curve $c(t) = \cR_x(tv)$ satisfies $c(0) = \x$ and $c'(0) = \v$ \cite[Ch3, Def3.47]{boumal2023introduction}.
There are different ways to perform retractions, such as exponential map and metric retractions~\cite[\textsection3.6]{boumal2023introduction}.
However, in general, to obtain the exponential map of a manifold, one requires to solve a differential equation.
Thus, for computational efficiency, we consider the metric retraction $
\cR_{\x}(\v) = \argmin_{\y \in E}\| \x + \v - \y \|_E^2 
~\st~ h(\y) = 0$ \cite[\textsection7.7]{boumal2023introduction}.
Note that $\cR_{\x}(\v)$ is possibly non-unique and hard to compute.
Among all the retractions we focus on
\begin{equation}\label{metric_retraction}
\cR_{\x}(\v) \coloneqq 
\dfrac{\x+\v}{\|\x+\v\|_E}.
\tag{Metric retraction}
\end{equation}

\noindent\textbf{Riemannian and Euclidean gradient.}
Given a (restricted) function $f\big|_{\cM} : \cM \to \IR$ with its smooth extension $f : E \to \IR$,
let $\nabla f (\x)$ be the Euclidean gradient of $f$ in the standard Euclidean basis at a point $\x \in E$, the Riemannian gradient of $f\big|_{\cM}$ defined on $\cM$, denoted as $\grad f\big|_{\cM}$, at a point $\z$, wrt. the reference point $\x$, is the Euclidean gradient $\nabla f(\z)$ projected onto $T_{\x}\cM$.
That is $\grad f\big|_{\cM} (\z)  = \proj_{T_{\x}\cM} \nabla f(\z) $.
For the sake of cheap computational cost, we endow the manifold with the standard Euclidean metric\footnote{The complete Riemannian gradient is computed using metric tensor.
Let $g$ be the metric tensor of $\cM$ and let $G_{\x}$ denotes the matrix representation of $g$,
then $\grad f\big|_{\cM} (\z) = G^{-1}_{\x} \proj_{T_{\x}\cM} \nabla f(\z)$.}.
To ease the notation, we write $f\big|_{\cM}$ as $f$ taking the natural extension.

\subsection{Riemannian Multiplicative Update (RMU)}\label{sec:tool:subsec:RMU}
In this work, we propose a RMU for solving Chordal-NMF in the form of $(\cP_1)$.
RMU is a modified  Riemannian gradient descent (RGD) update: $\x_{k+1} = \cR_{\x_k}( \alpha \v_k)$, where $\v_k = -\grad f(\x_k)$ is the (negative) Riemannian gradient of $f$ at $\x_k$.
RMU is projection-free, and it guarantees the nonnegativity of $\x$ by selecting a special element-wise stepsize such that $\x_{k+1}$ stays within $\cX$.
Below we review Euclidean Multiplicative Update (EMU) and then we generalize it to the Riemannian case.
\\

\noindent\textbf{Euclidean MU (EMU).}
EMU was first proposed in~\cite{richardson1972bayesian,lucy1974iterative,daube1986iterative}, and gained popularity in NMF~\cite{gillis2020nonnegative}.
EMU can be done via a ``element-wise sign decomposition'' of the Euclidean gradient $\nabla f(\x) = \nabla^+ f(\x) - \nabla^- f(\x)$, where $\nabla^+ f(\x) \geq \zeros$ and $\nabla^- f(\x) \geq \zeros$.
Let denote $\odot$ the element-wise product and $\oslash $ the element-wise division, for solving a nonnegative-constrained Euclidean optimization problem, 
the EMU step has the form
\begin{equation}\label{eq:E-MU}
\x_{k+1} 
~=~
\x_k 
\odot
\nabla^-f(\x_k)
\oslash 
\nabla^+f(\x_k),
\tag{EMU}
\end{equation}
which is obtained by choosing an element-wise stepsize $\balpha=\x_k \oslash \nabla^+f(\x_k)$ in the Euclidean gradient descent step $\x_{k+1}=\x_k-\alpha\nabla f(\x_k)$, see~\cite{yoo2008orthogonal} for the derivation\footnote{In the literature~\eqref{eq:E-MU} is written as $\x_{k+1}=\x_k\dfrac{\nabla^-f(\x_k)}{\nabla^+f(\x_k)}$.
We do not use this convention because it confuses with metric retraction for our purpose.}.
\\

\noindent\textbf{RMU.}
Now we generalize EMU.
First, a Riemannian gradient $\grad f(\x)$ always admits element-wise sign decomposition\footnote{For any object $\b$, it can be written as $\b=\b^+-\b^-$ with $\b^+=\max(\zeros,\b)$ and $\b^-=\max(\zeros,-\b)$.}
$\grad f(\x) = \grad^+ f(\x) - \grad^- f(\x)$.

We generalize \eqref{eq:E-MU} to the Riemannian case in Proposition~\ref{prop:RMU_update}, where we prove that the nonnegativity of the update is preserved by RMU without projection.
The key idea of the proof is the metric retraction and the choice of step-size $\balpha$ that acts component-wise on the update direction $\v_k$.

\begin{proposition}[RMU]\label{prop:RMU_update}
Denote $\v_k = - \grad f(\x_k)$ the anti-parallel direction of the Riemannian gradient of a manifold $\cM$ at $\x_k$, and let $\cR_{\x_k}$ be the metric retraction onto $\cM$.
If a nonnegative $\x_k$ is updated by $\x_{k+1} = \cR_{\x_k}( \balpha \odot  \v_k)$
with an element-wise stepsize $\balpha=\x_k \oslash \grad^+ f(\x_k) \in E$, then $\x_{k+1} \geq \zeros$ and is on $\cM$.
\end{proposition}
    \begin{proof}
The term $\balpha \odot  (-\grad f(\x_k))$ with $\balpha=\x_k \oslash \grad^+ f(\x_k)$ can be computed as 
\[
\begin{array}{lll}
-\balpha \odot  \grad f(\x_k) 
&=
\big( \x_k \oslash \grad^+ f(\x_k) \big)  \odot  \big( \grad^- f(\x_k) - \grad^+ f(\x_k) \big)
\\
&=
\x_k \odot \grad^- f(\x_k) \oslash \grad^+ f(\x_k) - \x_k .
\end{array}
\]
The element-wise operations as linear transformations  in general do not preserve the tangent condition of a manifold into a point.
However here we can still apply~\eqref{metric_retraction} on $-\balpha \odot  \grad f(\x_k)$, the $\x_k$ terms got canceled and gives
\[
\begin{array}{ll}
\x_{k+1} 
=
\cR_{\x_k}\big( -\balpha \odot  \grad f(\x_k)\big) 
&=
\dfrac{\x_k - \balpha \odot  \grad f(\x_k)}{\| \x_k - \balpha \odot  \grad f(\x_k) \|_E}
\\
&= 
\dfrac{ \x_k \odot \grad^- f(\x_k) \oslash \grad^+ f(\x_k)}
{
\|  \x_k \odot \grad^- f(\x_k) \oslash \grad^+ f(\x_k) \|_E
}.    
\end{array}
\]
The ratio between $\grad^+ f$ and $\grad^- f$ is nonnegative as its parts are nonnegative by definition, the denominator is nonnegative, therefore $\x_{k+1}$ is nonnegative if $\x_k$ is nonnegative.
\end{proof}
 
Proposition~\ref{prop:RMU_update} also gives a convergence condition to a critical point: 
\begin{itemize}
    \item If the solution lies strictly inside the nonnegative orthant, this corresponds to $\grad f (\x_k) = \zeros$, which is equivalent to $\grad^- f(\x_k) = \grad^+ f(\x_k)$. 
    That is, the convergence of the sequence $\{\x_k\}_{k \in \IN}$ can be easily checked by comparing $\grad^- f(\x_k)$ and $\grad^+ f(\x_k)$.

    \item Otherwise, in the constrained setting, convergence is characterized by the KKT conditions on $\cM \cap \IRn_+$.
\end{itemize}

As RMU is projection-free, it has a lower per-iteration computational cost, making it suitable for Chordal-NMF.

We are now ready to move on to the explicit update of $\H$ and $\W$.

\section{Subproblems in the algorithm}\label{sec:algo}
In this section, we discuss how to solve the subproblems in Algorithm~\ref{alg:BCD}.

\subsection{Column-wise h-subproblem over ellipsoid}\label{sec:H}
Here we discuss how to solve the h-subproblem.
We re-formulate the h-subproblem on $\h$ over the ellipsoid  $\cE_{\W^\top\W}^{r-1}$ as
\begin{equation}\label{h-manifold-subproblem}
\displaystyle \argmin_{\h}
\Big\{
\phi(\h) 
\coloneqq
1 - \langle\m_{:j}, \W\h\rangle
\Big\}
~\st~
\h \in \IRr_+ \cap \cE_{\W^\top\W}^{r-1},
\tag{h-manifold-subproblem}
\end{equation}
where we assume $\W$ has full rank, the ellipsoid $\cE^{r-1}_{\W^\top\W} \subset \IRr$ defined by a Positive Definite (PD) matrix $\W^\top\W \in \IR^{r\times r}$ is
\begin{equation}\label{def:ellipsoid-h}
\cE_{\W^\top\W}^{r-1}
\coloneqq
\Big\{ 
\bxi \in \IRr ~\big|~ 
\big\langle \bxi ,  \bxi \big\rangle_{\cE_{\W^\top\W}^{r-1}} 
\coloneqq 
\big\langle \W^\top \W  \bxi ,  \bxi \big \rangle = 1
\Big\} 
\subset \IRr.
\tag{Ellipsoid}
\end{equation}
To ease notation, sometimes we write $\langle  \bxi , \bzeta \rangle_{\cE_{\W^\top\W}^{r-1}} $ as $\langle  \bxi , \bzeta \rangle_{\cE} $ 
and $\|  \bxi \|_{\cE_{\W^\top\W}^{r-1}} $ as $\|  \bxi \|_{\cE}$.
We remark that:
\begin{itemize}
    \item As a subset of $\IRr$, the set $\cE_{\W^\top\W}^{r-1}$ is a smooth submanifold~\cite[Def3.10]{boumal2023introduction} and we can show that the inner product $\langle  \bxi , \bzeta \rangle_{\cE} $ is a Riemannian metric~\cite[Prop3.54]{boumal2023introduction}.

    \item The set $\IRr_+ \cap \cE_{\W^\top\W}^{r-1}$ is a nonsmooth manifold: $\IRr_+$ has sharp corners at the boundary. 
    See \textsection\ref{sec:tool} for discussion on the issues caused by such non-smoothness.
\end{itemize}

\noindent\textbf{Tools on ellipsoid manifold.}
The tools to solve \eqref{h-manifold-subproblem} by RMU is summarized in Table~\ref{table:manifold} and detail in the following paragraphs.
\begin{table}[h]
\centering
\caption{Summary of objects.
Top: for $\h$ and $\A = \W^\top \W$.
Bottom: For $\W$ and $\A_j = \h_{:j}\h_{:j}^\top$.
}
\label{table:manifold}
\renewcommand*{\arraystretch}{1.6}
\begin{tabular}{ll}
Name / Reference &   Definition / expression         
\\ \hline \hline
Manifold of $\h$ \eqref{def:ellipsoid-h}  & 
$\cE_{\A}^{r-1} \coloneqq 
\Big\{
\bxi \in \IRr ~\big|~ \langle \A \bxi, \bxi \big\rangle = 1
\Big\}
$
\\
Tangent space of $\cE_{\A}^{r-1}$ at $\bzeta$ 
\eqref{def:TangentEllipse}
&
$T_{\bzeta} \cE_{\A}^{r-1}  \coloneqq 
\Big\{
\bxi \in \IRr ~\big|~ \langle \A \bxi, \bzeta \big\rangle = 0
\Big\}$
\\
Project $\bxi$ onto $T_{\bzeta}\cE_{\A}^{r-1}$ (Proposition\,\ref{prop:projTE_h})  & 
$\proj_{T_{\bzeta}\cE_{\A}^{r-1}}(\bxi)
=
\bigg(
\I_r-
\dfrac{(\A\bzeta)^{\otimes2}}{\|\A\bzeta\|_2^2}
\bigg)\bxi
$
\\
Retraction (Proposition\,\ref{prop:restract_h})
& 
$\cR_{\bzeta}(\bxi) 
~=~ 
\dfrac{\bzeta + \bxi}{\| \bzeta + \bxi \|_\cE} 
$
\\ \hline 
Manifold of $\W$ \eqref{eq:matrix_man} &
$\cM_j \coloneqq \Big\{
\W \in \IRmr ~\big|~ \langle \W, \W \big\rangle_{\A_j^{1/2}} = 1
\Big\}
$
\\
Tangent space of $\cM_j$ at $\Z$ \eqref{def:TangentMatrixMan}
& 
$T_{\Z} \cM_j  \coloneqq 
\Big\{
\W \in \IRmr ~\big|~ \langle \W ,\Z \A_j \big\rangle_\textrm{F} = 0
\Big\}$
\\
Project $\W$ onto $T_{\Z}\cM_j$ (Proposition\,\ref{prop:orth_projMatrixMan}) & 
$\proj_{T_{\Z}\cM_j}(\W)
= 
\bigg(\I-
\dfrac{ (\Z\A_j)^{\otimes 2}}{
\|\Z\A_j\|_\textrm{F}^2}\bigg) 
\W
$
\\
Retraction (Proposition\,\ref{prop:retract_Mj})
&
$
\cR_{\Z}(\W)
= 
\dfrac{\Z+\W}{\|\Z+\W\|_{\A_j^{1/2}}
}
$
\end{tabular}
\end{table}

\noindent\textbf{Tangent space.}
Let $\A \coloneqq \W^\top\W$ and the function $h(\bxi) = \langle \A\bxi, \bxi \rangle - 1$, we denote $\cE_{\A}^{r-1}$ as the manifold and its differential $\textrm{D}h(\bxi)[\bzeta] = 2 \langle \A\bxi, \bzeta \rangle$.
The tangent space of $\cE_{\A}^{r-1}$ at a reference point $\bzeta \in \cE_{\A}^{r-1}$, is defined as the set
\begin{equation}\label{def:TangentEllipse}
T_{\bzeta}\cE_{\A}^{r-1}
\coloneqq  \{ \bxi \in \IRr ~|~ \langle \A\bzeta, \bxi \rangle  = 0 \}
= \textrm{Ker\,D}h(\bzeta).
\end{equation}

\noindent\textbf{Projection onto tangent space, and Retraction.}
Consider $\textrm{D}h(\bxi)[\bzeta]$, we define its adjoint operator as $\textrm{D}h(\bxi)^*[{\alpha}]=2{\alpha}\A\bxi$, 
where $\alpha$ takes the value $\overline{\alpha}= \frac{1}{2}\frac{\langle\A\bzeta,\bxi\rangle}{\|\A\bzeta\|_2^2}$ obtained by solving the least squares problem $\overline{\alpha} =  \displaystyle \argmin_{\alpha\in\IR}{\big\| \bxi-\textrm{D}h(\bzeta)^*[\alpha] \big\|_2^2}$.
Now by the previously defined $\textrm{D}h(\bxi)^*[{\alpha}]$ and $\overline{\alpha}$, 
we have the projector 
$
\proj_{T_{\bzeta}\cE_{\A}^{r-1}}(\bxi)
=
\bxi-\textrm{D}h(\bzeta)^*[\overline{\alpha}]
$ 
as shown in the following proposition.
\\

\begin{proposition}\label{prop:projTE_h}
The map $\proj_{T_{\bzeta}\cE_{\A}^{r-1}}:\IRr\rightarrow T_{\bzeta}\cE_{\A}^{r-1}$ that 
bring $\bxi$ onto the tangent space at the reference point, $\bzeta$, defined as 
\begin{equation}\label{prop:orth_proj}
\proj_{T_{\bzeta}\cE_{\A}^{r-1}}(\bxi)
=
\bxi-\dfrac{\langle\A\bzeta,\bxi\rangle}{\|\A\bzeta\|_2^2}\A\bzeta,
\tag{proj-ellips}
\end{equation}
is an orthogonal projector.
\end{proposition}
\begin{proof}
Using formula \eqref{formula:tensor}, we have \eqref{prop:orth_proj} equals to the following matrix form
\[
\bxi-\dfrac{\langle\A\bzeta,\bxi\rangle}{\|\A\bzeta\|_2^2}\A\bzeta 
~=~
\bigg(
\I_r-
\dfrac{(\A\bzeta)^{\otimes 2}}{\|\A\bzeta\|_2^2}
\bigg)\bxi
~\eqqcolon~ \P \bxi
.
\]
To show that \eqref{prop:orth_proj} is a projector onto the tangent space of the reference point $\bzeta$ we use the definition of orthogonal projector in \textsection\ref{sec:tools:subsec:background}.

First, the range of $\proj_{T_{\bzeta}\cE_{\A}^{r-1}}$, denoted as
$\textrm{Im}\big(
    \proj_{T_{\bzeta}\cE} 
    \big)$, is exactly $T_{\bzeta}\cE_{\A}^{r-1}$, since
\[      
\bxi = 
\bxi-\dfrac{\langle\A\bzeta,\bxi\rangle}{\|\A\bzeta\|_2^2}\A\bzeta 
~\iff~
    \dfrac{\langle\A\bzeta,\bxi\rangle}  {\|\A\bzeta\|_2^2}  \A\bzeta
    = 
    \zeros
~\overset{\A\bzeta\neq\zeros}{\iff}~
\bxi \in T_{\bzeta}\cE_{\A}^{r-1}
.
\]
Second, we show that $\proj_{T_{\bzeta}\cE_{\A}^{r-1}}$ is idempotent by showing $\P$ is idempotent
\[
\P^2
=
\I_r
- \dfrac{(\A\bzeta)^{\otimes 2}}{\|\A\bzeta\|_2^2}
~ \underbrace{-\, \dfrac{(\A\bzeta)^{\otimes 2}}{\|\A\bzeta\|_2^2}
+ 
\dfrac{(\A\bzeta)^{\otimes 2}}{\|\A\bzeta\|_2^2} \dfrac{(\A\bzeta)^{\otimes 2}}{\|\A\bzeta\|_2^2}
}_{=\zeros}
= \P.
\]
We show $\proj_{T_{\bzeta}\cE_{\A}^{r-1}}$ is orthogonal by showing $\P$ is orthogonal.
As $\A$ is symmetric, so $\P^\top = \P$.
Being symmetric and idempotent, $\P$ is orthogonal \cite[\textsection0.9.13]{horn2012matrix}.
\end{proof}

\begin{proposition}[metric retraction]\label{prop:restract_h}
The map $\cR_{\bzeta}:T_{\bzeta}\cE_{\A}^{r-1}\rightarrow \cE_{\A}^{r-1}$ defined as 
\begin{equation}\label{prop:retract_ellipse}
\cR_{\bzeta}(\bxi) 
~=~
\dfrac{\bzeta + \bxi}{\| \bzeta + \bxi \|_\cE}
\overset{T_{\bzeta}\cE_{\A}^{r-1}}{~=~} 
\dfrac{\bzeta + \bxi}{\sqrt{1 + \| \bxi\|_\cE^2}},
\tag{$\cR$-ellipsoid}
\end{equation}
is a metric retraction for the manifold $\cE_{\A}^{r-1}$.  
\end{proposition}
We give the proof of the proposition in Appendix \ref{app:proof_retract_h}.

\noindent\textbf{Riemannian Gradient.}
We solve \eqref{h-manifold-subproblem} using RMU discussed in \textsection\ref{sec:tool:subsec:RMU}.
The Riemannian gradient 
$\grad \phi(\h_k) = \proj_{T_{\h_k}\cE} \nabla \overline{\phi}_{\cE}(\h_k)$, is
\[
\grad \phi(\h_k)
=
\bigg(\I_r-\dfrac{(\A\h_k)^{\otimes 2}}{\|\A\h_k\|_2^2}\bigg)
\Big(- \W^\top \m_{:j} \Big)
= 
\underbrace{
\dfrac{(\A\h_k)^{\otimes 2}}{\|\A \h_k\|_2^2} \W^\top\m_{:j}}_{\grad^+ \phi}
- 
\underbrace{
\W^\top
\m_{:j}
}_{\grad^- \phi},
\]
where we used both: \eqref{prop:orth_proj}, \eqref{formula:tensor}, the fact that $\h_k \in \cE_{\A}^{r-1}$, and the Euclidean gradient of $\h_k$ through Proposition\,\ref{prop:grad_f}.

Then, the RMU under metric retraction discussed in \textsection\ref{sec:tool:subsec:RMU} is 
\[
\h_{:j,k+1} = \dfrac{\z}{\| \z \|_{\cE_{\A}^{r-1}}},
~\text{where}~
\z = \h_{:j,k} \odot \grad^- \phi(\h_{:j,k})[\h_{:j,k}] \oslash  \grad^+ \phi(\h_{:j,k})[\h_{:j,k}]
.
\]

\subsection{The finite sum matrix-wise W-subproblem}\label{sec:W}
Now we discuss how to solve the W-subproblem in Algorithm~\ref{alg:BCD}.
Let $\langle \cdot,\cdot\rangle_\textrm{F} = \tr(\cdot^\top \cdot)$ be the Frobenius inner product, $\B_j \coloneqq \m_{:j} \h_{:j}^\top$ and $\A_j \coloneqq \h_{:j}\h_{:j}^\top$ (not the $\A,\B$ defined in \textsection\ref{sec:H}),
we rewrite the W-subproblem as
\begin{equation}\label{W-subproblem}
\displaystyle \argmax_{\W\geq \zeros}
\bigg\{
F(\W) =
\sum_{j=1}^n
\dfrac{\langle \B_j , \W \rangle_\textrm{F}}
{\sqrt{\langle \W, \W \A_j \rangle_\textrm{F}}}
\bigg\}.
\tag{W-subproblem}
\end{equation}
Note that if a column $\h_j$ is zero, the corresponding term is removed in the sum.

To solve \eqref{W-subproblem}, we first consider $n=1$.
We have a manifold which we name ``shell of a single twisted spectrahedron''.
Table~\ref{table:manifold} summarizes the results of RMU on such manifold.
See Appendix \ref{sec:W:subsect:singleW} for the derivation of manifold $\cM_j$, the tangent space $T_{\Z}\cM_j$, the projection and the retraction.

\begin{lemma}[RMU update on $\W$]\label{RMU_W_n1}
If $n=1$, we update $\W$  as 
\begin{equation}\label{updt:RMU_W_n1manifold}
\W_{k+1} =
\dfrac{\W_k \odot \grad^- f_j(\W_k)[\W_k] \oslash  \grad^+ f_j(\W_k)[\W_k]}{
\Big\| \W_k \odot \grad^- f_j(\W_k)[\W_k] \oslash  \grad^+ f_j(\W_k)[\W_k]\Big\|_{\A_j^{1/2}}
},
\end{equation}
where 
$
\grad f_j(\W) =
\underbrace{
\dfrac{(\Z\A_j)^{\otimes 2} }{\| \Z \A_j\|_\textrm{F}^2}  \B_j    
}_{\grad^+ f_j(\W)}
-
\underbrace{\B_j 
}_{\grad^- f_j(\W)}.
$
\end{lemma}

For \eqref{W-subproblem} with $n \geq 2$, having multiple fractions makes the problem hard to solve. 
First, there is no a single manifold to guarantee $\langle \W, \W \A_j \rangle_\textrm{F}=1$ for all the denominators.
Second, the computation cost of manifold approach for $n \geq 2$ is high: it takes $\cO(2m^2)$ cost to compute the adjoint, with a projection step that has no closed-form solution.
Details are reported in \textsection\ref{sec:app:efficient_RgradW:n2} of the Appendix.

We do not solve \eqref{W-subproblem} with $n\geq2$ using RMU (which is designed for $n=1$), For illustration, we also considered a consensus-averaging approach between all the manifolds, which we call avgRMU. 
The motivation was that the manifolds can conflict with each other in RMU, so averaging provides a natural ``relaxation'' to fit RMU in the intersection of the manifolds.
We remark that later we found that experiments show that this approach performs worse than the Euclidean method, and we do not use it as our main solution.

Furthermore, we try two alternative methods.

\noindent\textbf{EPG} for \eqref{W-subproblem}
$\W_{k+1} = [\W_k-\eta \nabla F(\W_k)]_+$ with a stepsize $\eta\geq 0$.
In the update, we compute the gradient by quotient rule, $\nabla \langle \W, \W \A_j   \rangle_\textrm{F} = \W \A_j + \W \A_j^\top$, $\A_j^\top = \A_j$ and  Proposition~\ref{prop:grad_f} as in the following proposition.
\begin{proposition}\label{prop:gradient_W}
The Euclidean gradient for $F(\W)$ in \eqref{W-subproblem} is
\begin{equation}\label{gradient_W}
\nabla F(\W) 
~=~
\sum_{j=1}^n
\dfrac{\B_j  }{\langle \W,  \W \A_j \rangle^{1/2}_\textrm{F}}
-
\dfrac{ \langle \B_j , \W \rangle_\textrm{F}   \W \A_j}{\langle \W,  \W \A_j  \rangle^{3/2}_\textrm{F}}
.
\end{equation}
\end{proposition}

\noindent\textbf{Fractional Programming (FP).}\label{par:fracprog}
We solve \eqref{W-subproblem} using FP by the Dinkelbach transform~\cite{dinkelbach1967nonlinear} and Jagannathan's iteration~\cite{jagannathan1966some} with the EPG.

\section{Experiment}\label{sec:Exp}
In this section we first perform testing on verifying the effectiveness of RMU on the subproblems, then we test the performance of Chordal-NMF on synthetic and real datasets.

\subsection{RMU on h-subproblem and W-subproblem}

\noindent\textbf{Test on updating H.}
We compare RMU and methods mentioned in \textsection\ref{sec:tool:subsec:review} on synthetic datasets randomly generated under normal distribution $\cN(\zeros_r, \I_r)$ with negative entries replaced by zero.
We performed Monte Carlo trials on datasets with different settings of $(m,r)$.
Fig.~\ref{fig:h_speed} shows a typical convergence of the objective function. 
RMU is among the fastest methods with strict feasibility. 
RADMM (Riemannian ADMM in \textsection\ref{sec:tool:subsec:review}) is also fast, while EPG and RALM (Riemannian Augmented Lagrangian multiplier method) has the worst convergence.
In the test, we run at most $10^{5}$ iterations, or stop the algorithm when the $\|\h_{j,k}-\h_{j,k-1}\| \leq 10^{-12}$.

\begin{figure}[!ht]
\begin{minipage}{0.51\textwidth}
\includegraphics[width=\textwidth]{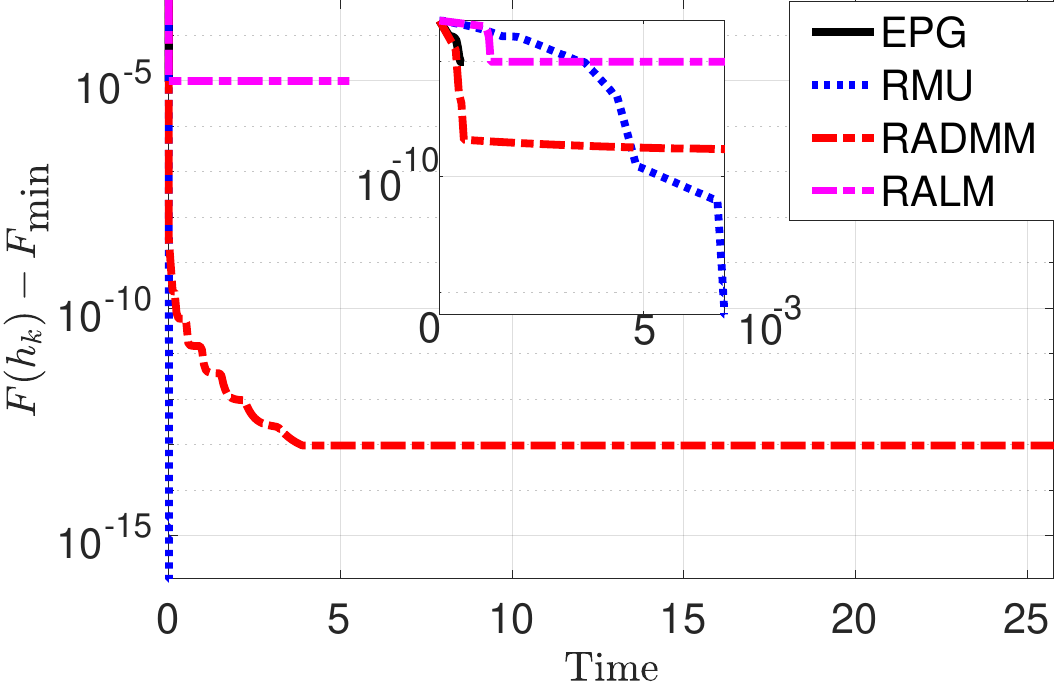}
\end{minipage}
\begin{minipage}{0.4\textwidth}
\begin{tabular}{l|l}
\hline
$(m,r)=(5,3)$ & Rankings \\ \hline 
EPG & (47, 20, 26, 7, 0)
\\
RMU & (53, 47, 0, 0, 0)
\\
RADMM & (0, 7, 16, 0, 77)
\\
RALM & (0, 26, 31, 14, 29)
\\ \hline 
$(m,r)=(100,10)$ & Rankings \\ \hline 
EPG & (5, 87, 8, 0, 0)
\\
RMU & (95, 5, 0, 0, 0)
\\
RADMM & (0, 5, 0, 0, 95)
\\
RALM & (0, 3, 61, 5, 31)
\\ \hline
\end{tabular}
\end{minipage}
\caption{
Left: A typical convergence of the four methods: EPG, RMU, RADMM and RALM. 
In the y-axis, $F_{\min}$ refers to the smallest achieved $F$ value across the methods.
The smaller subplot
is zooming the first $7\times 10^{-3}$ seconds.
In this test, we have $(m,r)=(5,3)$ and RMU is the fastest methods with strict feasibility. 
Right: the results on different settings of $(m,r)$.
A ranking vector $(x,y,z,...)$ means the method has $x$ times out of 100 runs having the best performance, $y$ times  out of 100 runs having the 2nd performance, and so on.
There are four methods, and the fifth number indicate how many times the method produces an infeasible solution ( having negative values) in the end.
}
\label{fig:h_speed}
\end{figure}

\noindent\textbf{Test on updating W.}
As RMU for the whole sum $F(\W)$ in \eqref{W-subproblem} is computationally infeasible, we consider averaging the RMU for each manifold.
We compare the performance of three methods: EPG, avgRMU and FP. 
We run experiments on two synthetic datasets randomly generated under zero-mean unit variance Gaussian distribution with negative entries replaced by zero. 
We performed 100 Monte Carlo runs on two datasets sizing $(m,n,r) = (10,25,4)$ and $(m,n,r) = (100,100,4)$. 
Fig.~\ref{fig:W_speed} shows the typical convergence of three methods.
Generally EPG is the fastest method.
For $(m,n,r)=(10,25,4)$, 99 times out of 100 MC runs EPG is the fastest and 96 times out of 100 runs FP is the 2nd fastest.
For $(m,n,r)=(100,100,4)$, 93 times out of 100 MC runs EPG is the fastest and 93 times out of 100 runs FP is the 2nd fastest.
avgRMU is the slowest in all the 100 runs for both cases.

\begin{figure}[!ht]
\centering
\includegraphics[width=0.495\textwidth]{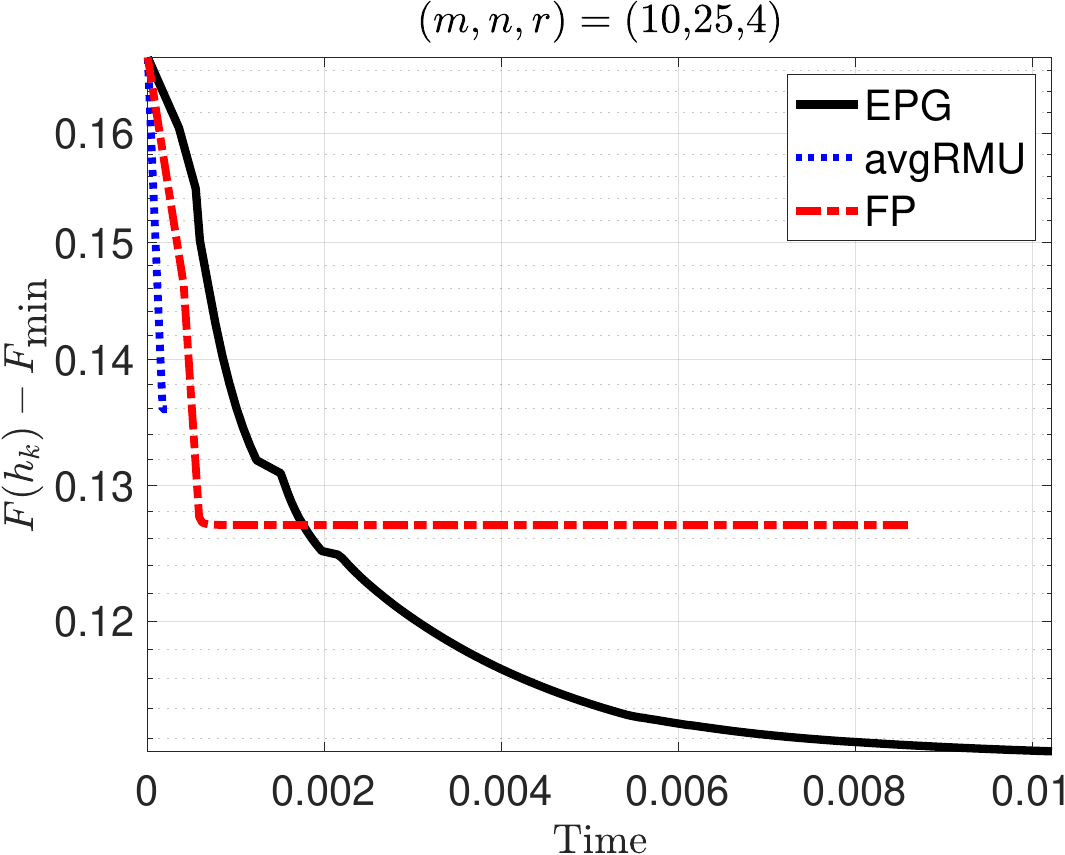}
\includegraphics[width=0.495\textwidth]{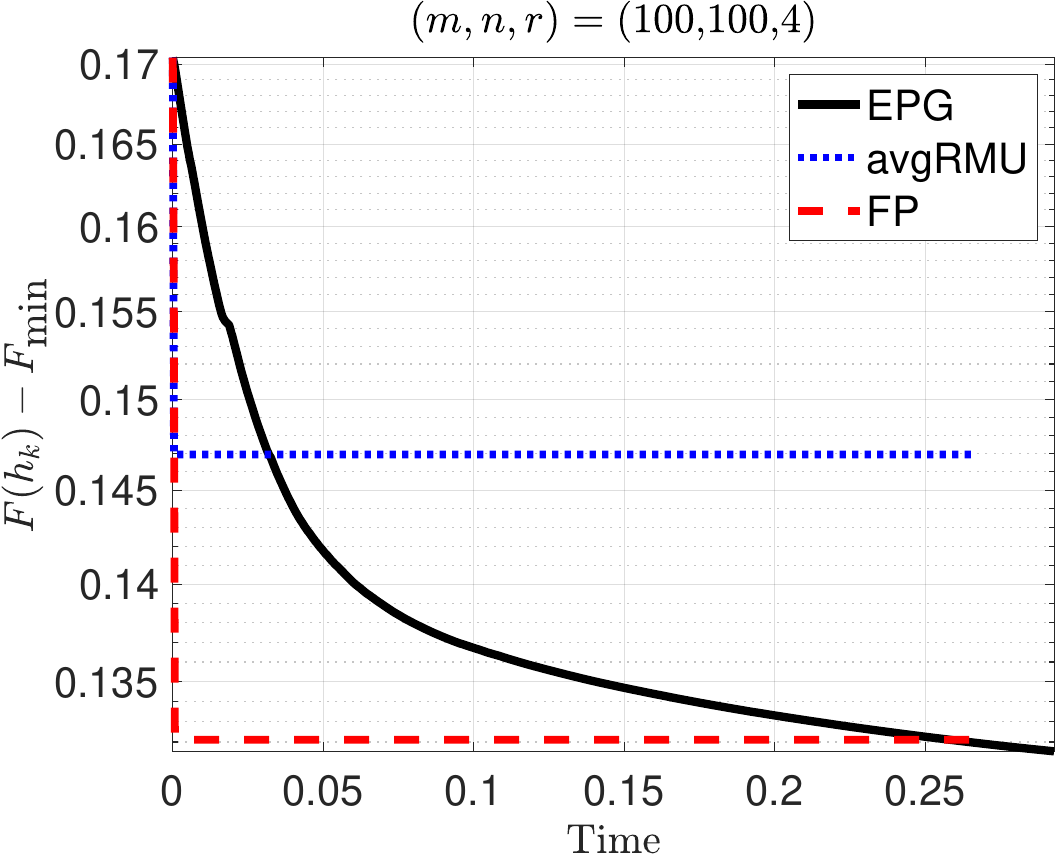}
\caption{Typical convergence of three methods: EPG, avgRMU and FP.
}
\label{fig:W_speed}
\end{figure}

\subsection{Testing Chordal-NMF on synthetic dataset}

\noindent\textbf{How Chordal-NMF is applied.}
Given a data matrix $\M$, we remove all zero columns, then we normalize each column by its $\ell_2-$norm.
After that we run Chordal-NMF on $\M$ and get the decomposition $\W\H$.
According to the preliminary test results, 
for the h-subproblem, we perform RMU (\textsection\ref{sec:H}), and 
for the W-subproblem, we perform EPG (\textsection\ref{sec:W}).
\\

\noindent\textbf{Benchmark.}
We compare Chordal-NMF with the classical F-norm NMF (FroNMF) based on the algorithm HALS~\cite[Ch8.3.3]{gillis2020nonnegative}.
We set the rank of Chordal-NMF equal to the one in FroNMF.
In the experiments, all the methods start with the same initialization, where elements in $\W_0, \H_0$ are generated under uniform distribution $\cU[0,1]$.
The experiments were conducted in MATLAB 2023a\footnote{The MALTAB code is available at \url{https://github.com/flaespo/Chordal_NMF}.} 
in a machine with OS Windows 11 Pro on a Intel Core 12 gen. CPU 2.20GHz and 16GB RAM.
\\

\noindent\textbf{Synthetic dataset.} 
We use a dataset $\M^{\epsilon,\delta} = \W_{\text{true}} \H_{\text{true}}^{\epsilon,\delta}$ with $\epsilon >0, \delta >0$ as
\[
\W_{\text{true}}
=
\begin{bmatrix}
 0.8 & 0.1 & 0.1
 \\
 0.1 & 0.8 & 0.1
 \\
 0.1 & 0.1 & 0.8
\end{bmatrix},
\H_{\text{true}}^{\epsilon,\delta}
=
\begin{bmatrix}
1-\epsilon & \delta( 1-\epsilon) & \epsilon   & \delta\epsilon   & \epsilon   & \delta\epsilon
\\
\epsilon   & \delta\epsilon   & 1-\epsilon & \delta(1-\epsilon) & \epsilon   & \delta\epsilon 
\\
\epsilon   & \delta\epsilon   & \epsilon   & \delta\epsilon   & 1-\epsilon & \delta(1-\epsilon) 
\end{bmatrix}.
\]
The matrix $\W_{\text{true}}$ represents a cone in $\IR^3$ (see Fig.\,\ref{fig:grid_synthetic}).
The matrix $\H_{\text{true}}^{\epsilon,\delta}$ represents how we generate the six columns of $\M^{\epsilon,\delta}$ by conic combination of columns of $\W_{\text{true}}$ under a small perturbation $\epsilon$ and an attenuation $\delta$.
The perturbation $\epsilon$ represents how much the columns of $\M^{\epsilon,\delta}$ deviate from the columns of $\W_{\text{true}}$, while the attenuation $\delta$ represents how much the columns $\M_{:2}, \M_{:4}, \M_{:6}$ in $\M^{\epsilon,\delta}$ have their norm scaled downward.
For $\delta$ getting smaller, it is getting harder for FroNMF to recover $\H_{\text{true}}$, while it is less a problem for Chordal-NMF since the angle between the data columns are invariant to the attenuation $\delta$.
See Fig.\,\ref{fig:grid_synthetic} for an illustration.

We construct $\M^{\epsilon,\delta}$ across different values of $(\epsilon,\delta)$, and run Chordal-NMF and FroNMF on each instance of $\M^{\epsilon,\delta}$ generated.
Then we extract the matrix $\H$ produced by the last iteration of the each method, and calculate the relative error between $\H$ and $\H_{\text{true}}^{\epsilon,\delta}$ as $\|\H-\H_{\text{true}}^{\epsilon,\delta}\|_F/\|\H_{\text{true}}^{\epsilon,\delta}\|_F$.
Fig.~\ref{fig:grid_synthetic} shows the heat-map of the  results across different values of $(\epsilon,\delta)$, showing that Chordal-NMF on average has a better recovery of $\H$ than FroNMF, especially for the case when $\delta$ is small.
We also compute the relative error on the matrix $\W$, where we found that the result on Chordal-NMF gives exactly the same performance as the one obtained from the FroNMF\footnote{Both Chordal-NMF and FroNMF gives a $4.02\%$ relative error on $\W$.}.
\begin{figure}[!h]
\centering
\includegraphics[width=.32\textwidth]{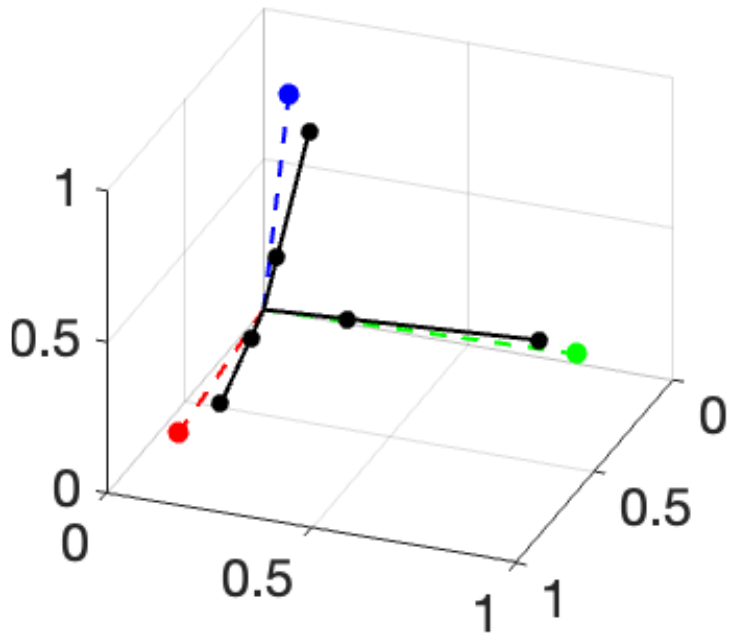}
~~~
\includegraphics[width=.6\textwidth]{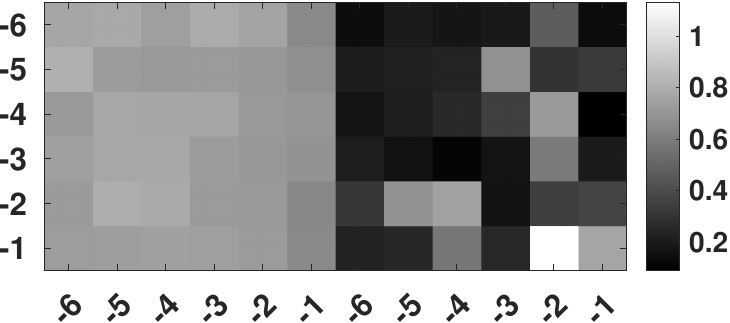}
\caption{
\textbf{Left}: the plot of $\W_{\text{true}}$ (the red, blue, green ray) and $\M(0.1,0.3)$ (the black rays).
\textbf{Right}:
The relative error $\|\H-\H_{\text{true}}^{\epsilon,\delta}\|_F/\|\H_{\text{true}}^{\epsilon,\delta}\|_F$ for the two methods, where left gird (the first six columns) is the result from FroNMF, and the right grid (the last six columns) is the result from Chordal-NMF.
In the grid, the x-axis is the value of $\delta$ (in log-scale) and the y-axis is the value of $\epsilon$ (in log-scale)
.
}
\label{fig:grid_synthetic}
\end{figure}

\subsection{Testing Chordal-NMF on real-world data}
NMF finds applications in real-world data analysis~\cite{gillis2020nonnegative}. 
An example is Earth Observation (EO) in remote sensing. 
EO typically involves the utilization of hyperspectral images, which are stored as matrices and can be effectively analyzed through NMF for unmixing purposes~\cite{feng2022hyperspectral}. 
A conventional approach in this context involves addressing the standard NMF problem by minimizing the point-to-point F-norm~\cite{ang2019algorithms}. 
However, in this section, we introduce a comparative analysis between the performance of FroNMF and the Chordal-NMF proposed in this work.
\\

\noindent\textbf{The EO Dataset.}
We chose a cloudy multispectral image due to its relevance in highlighting the advantages of employing Chordal-NMF compared to standard FroNMF for analyzing multispectral images under various cloudiness conditions. 
Our goal is to provide an illustration of how Chordal-NMF can better manage the presence of cloud conditions, thereby presenting itself as a useful alternative for conducting image analysis for EO. 
Here the matrix $\M$ is a cloudy multispectral image in an area of Apulia region in Italy, from the Copernicus space ecosystem\footnote{Pixel size rows:150, cols:290 and bands:12. 
Data from \url{https://dataspace.copernicus.eu/}}.
We use a reference image obtained in cloudless conditions as the ground truth\footnote{The cloudy image and the cloudless image are obtained in the same condition (except the cloudiness) 
with a temporal difference of 5 days.
}. \\

\noindent\textbf{The reconstruction.}
We run FroNMF and Chordal-NMF on the cloudy data $\M$ with the same initialization $\W_0,\H_0$ obtained from random uniform distribution $\cU[0,1]$.
We run Chordal-NMF and FroNMF with 5000 iterations.
For FroNMF, we run the implementation from \cite{gillis2020nonnegative}.
For Chordal-NMF, we run BCD with 25 iterations for each column on $\h_{:,j}$ to update $\H$, and 1 iteration of matrix-wise EPG update for $\W$.
The values of the Chordal function for the methods are: 0.1906100 (Initial Value), 0.0014398 (Chordal-NMF), 0.001279 (FroNMF).
It took ChordalNMF 2775 seconds (about 0.5 second per iteration), while it took FroNMF 2 seconds.
Fig.~\ref{fig:real_RGB} shows an RGB representation of the ground truth (cloudless reference image),
the cloudy data $\M$, and the reconstruction obtained from Chordal-NMF and FroNMF on $\M$.
\\
 
\begin{figure}[!h]
\centering
\includegraphics[width=1\textwidth]{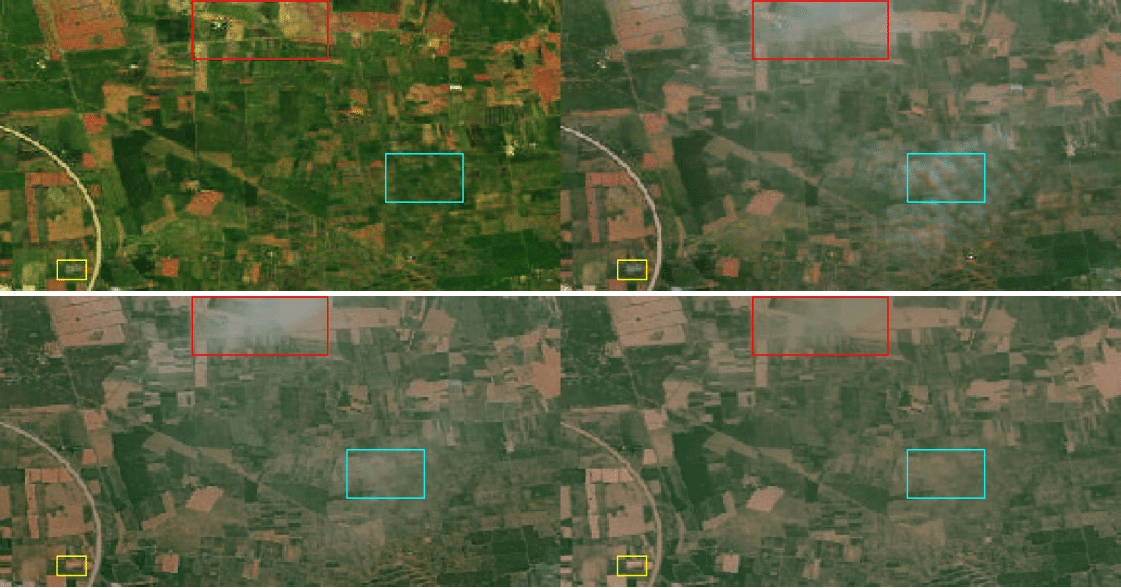}
\caption{
The RGB image of a scene.
Top left: the cloudless reference image.
Top right: the cloudy data image.
Bottom left: the reconstruction obtained by FroNMF.
Bottom right: the reconstruction obtained by Chordal-NMF.
In the images, the three color-boxes are selected areas under different cloundiness in which we perform further quantitative analysis.
}
\label{fig:real_RGB}
\end{figure}

\noindent\textbf{Further analysis on the three selected areas.}
In Fig.~\ref{fig:real_RGB}, three boxes were chosen to represent three distinct cloud conditions: the red box is a cloudy area, the cyan box is a less-cloudy areas, and the yellow box is a cloudless area.
We compute the spectral signatures of the pixels in these areas, and their mean behavior is plotted in Fig.~\ref{fig:real_spectra}.
We numerically compared these spectral profile vectors by two criteria:
\begin{enumerate}
\item \textbf{SID-SAM between the spectral profile vectors.}
Given two spectral profile vectors $\t = [t_1,...,t_N]$ and $\r = [r_1,...,r_N]$, 
we compute $\p = \t/\|\t\|_2$, $\q = \r /\|\r\|_2$, and then we compute the SID-SAM~\cite{du2004new} of the spectral information divergence (SID) and the spectral angle mapper (SAM), defined as $\text{SID-SAM} =  \text{SID}\times \tan{\alpha}$,
where
\[
     \text{SID}  
     = \sum\limits_{i=1}^N{p_i\log{\frac{p_i}{q_i}}}
     + \sum\limits_{i=1}^N{q_i\log{\frac{q_i}{p_i}}},
\quad
\alpha 
= 
\cos^{-1}
\left(
\sum\limits_{i=1}^N{t_ir_i} \bigg/ \sqrt{
\bigg( \sum\limits_{i=1}^N{t_i^2} \bigg)
\bigg( \sum\limits_{i=1}^N{r_i^2} \bigg)
}
\right)
.
\]
 \item \textbf{$\ell_2$-norm of the difference between the spectral profile vectors.}
 That is, $\|\t-\t_{\text{ChordalNMF}}\|_2$ and $\|\t-\t_{\text{FroNMF}}\|_2$.
 
\end{enumerate}
Results are reported in Table~\ref{tab:val_box} for the three areas, in which Chordal-NMF achieves a better performance.

\begin{figure}[!ht]
\centering
\includegraphics[width=0.99\textwidth]{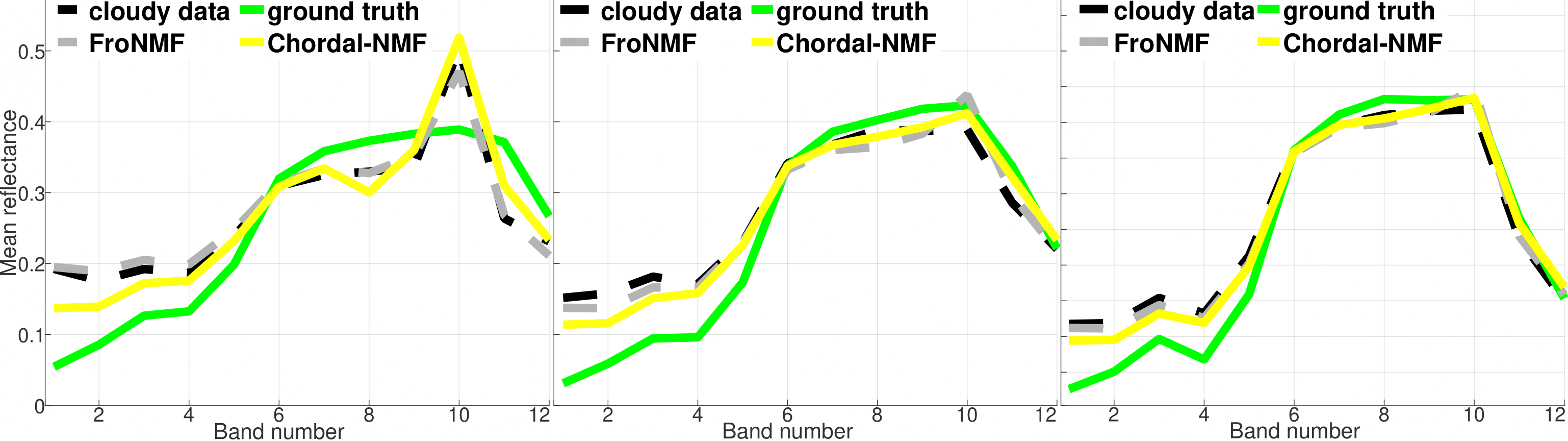}
\caption{
From left to right: the spectral signatures of the pixels in the 
red/cyan/yellow box in Fig.\ref{fig:real_RGB}.
In all the three cases, the spectral profile of obtained by Chordal-NMF is on average closer to the cloudless reference spectral profile.
}
\label{fig:real_spectra}
\end{figure}
\begin{table}[!ht]
\caption{The numerics of the pixels in the three boxes.}
\centering
\begin{tabular}{l||cc||cc||cc}
& \multicolumn{2}{c||}{Red box}  & \multicolumn{2}{c||}{Cyan box}  & \multicolumn{2}{c}{Yellow  box}
\\ \hline
Compared with cloudless &  SID-SAM & $\ell_2$ & SID-SAM & $\ell_2$  & SID-SAM & $\ell_2$     
\\\hline
Cloudy data &  2.2559 & 0.2655   &    0.0633  & 0.2154     & 0.0807 & 0.1594  
\\
FroNMF      &  2.0214  & 0.2650  &    0.5518 & 0.1901     & 0.3574 & 0.1458  
\\
Chordal-NMF  & 0.2090 &  0.2116 & 0.1486 & 0.1487        & 0.1161 &  0.1161         
\end{tabular}
\label{tab:val_box}
\end{table}

On the cloudy image recovery, we see that Chordal-NMF is always better than FroNMF regardless of the cloudiness of the image by achieving a lower SID-SAM value and a lower $\ell_2$-norm value.
For example, in the red box, Chordal-NMF seems to have a better recovery of the region under the cloud.
\\

\noindent\textbf{Performance on Samson dataset}
We test our approach also on a benchmark dataset for EO applications: Samson dataset~\cite{gillis2020nonnegative}. 
In this image, there are $95\times 95$ pixels, each pixel is recorded at $156$ channels. 
This dataset is not challenging as many analyses have been already carried out~\cite{ang2019algorithms}. 
Even if it is well known from the literature that there are three targets in the image, we want to highlight how Chordal-NMF is better at extracting the rock/soil component.
Fig.~\ref{fig:real_samson} shows the abundance map from the FroNMF and Chordal-NMF. 
We can see that the Chordal-NMF is able to recover rock under shadow water regions near the coast better than FroNMF.
\begin{figure}[!ht]
\centering
\includegraphics[width=.27\textwidth]{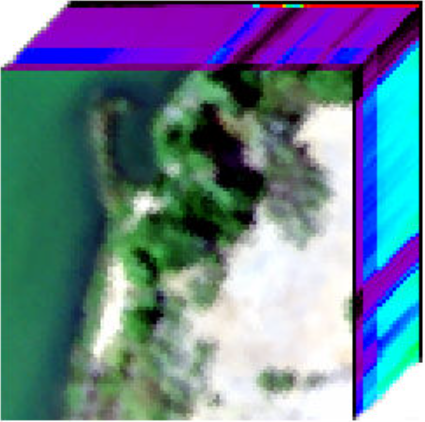}
~~
\includegraphics[width=.6\textwidth]{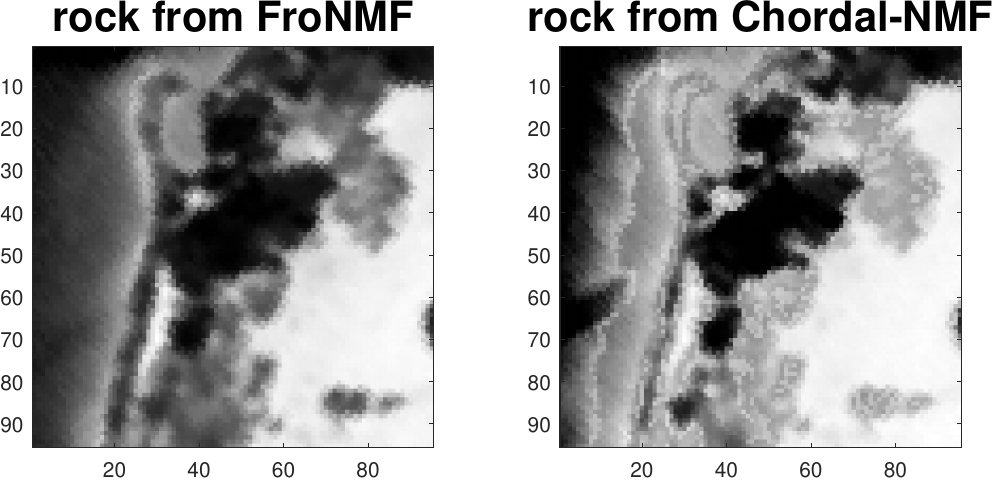}
\caption{The Samson dataset and the rock abundances map of the decompositions. 
From left to right: the data, the abundance maps obtained form FroNMF and  Chordal-NMF.
}
\label{fig:real_samson}
\end{figure}

\section{Conclusion}\label{sec:conc}
In this work, we introduced a NMF model called Chordal-NMF, which is different from the classical NMF using a point-to-point Euclidean distance, Chordal-NMF uses a ray-to-ray distance. 
Based on the geometric interpretation that NMF describes a cone, we argued that chordal distance, which measures the angle between two vector in the nonnegative orthant, is more suitable than the Euclidean distance for NMF.
Under a BCD framework, we developed a new projection-free algorithm, called Riemannian Multiplicative Update to solve the Chordal-NMF, where Riemannian optimization technique is used.
We showcase the effectiveness of the Chordal-NMF on the synthetic dataset and real-world multispectral images.

\appendix

\section{The proof of Proposition 1}
\begin{proof}
Let $h(\x) = (g\circ f)(\x) = \big(f(\x) \big)^2$ where  $g : \IR \to \IR : x \mapsto x^2$ is  convex, increasing (on $\IR_+$) and differentiable, then by chain rule 
$\partial (g\circ f)(\x) =
g'(f(\x)) \partial f(\x)$, which gives
\[
\nabla h(\x) 
~=~
2f(\x) \nabla f(\x) 
~=~
\dfrac{\langle \A \x + \b, \c \rangle}{2\| \D \x + \e \|_2} \nabla f(\x).
\tag{chain-rule}
\]
Now by the definition $h(\x) = \big(f(\x)\big)^2$, so
\[
\nabla h(\x) = \nabla \Bigg( \dfrac{\Big(\langle \A \x + \b, \c \rangle\Big)^2}{\| \D \x + \e \|_2^2}
\Bigg).
\tag{grad-h}
\]
Equate (chain-rule) and (grad-h) gives 
$
\nabla f(\x) 
= 
\dfrac{\| \D \x + \e \|_2}{\langle \A \x + \b, \c \rangle}
\nabla \bigg( \dfrac{\big(\langle \A \x + \b, \c \rangle\big)^2}{2\| \D \x + \e \|_2^2}
\bigg)
$.
By assumption $\D \x + \e \neq \zeros$, we use quotient rule $\nabla \dfrac{f}{g} = \dfrac{g \nabla f - f \nabla g}{g^2}$ to arrive at
\[
\begin{array}{rcl}
\nabla f(\x) 
&=&\displaystyle 
\dfrac{\| \D \x + \e \|_2}{\langle \A \x + \b, \c \rangle}
\dfrac{
\| \D \x + \e \|_2^2
\nabla \big(\langle \A \x + \b, \c \rangle\big)^2
-
\big(\langle \A \x + \b, \c \rangle\big)^2
\nabla \| \D \x + \e \|_2^2
}
{2\| \D \x + \e \|_2^4
}
\vspace{0.5mm}
\\
&=&
\displaystyle 
\dfrac{
\| \D \x + \e \|_2^2
\nabla \big(\langle \A \x + \b, \c \rangle\big)^2
-
\big(\langle \A \x + \b, \c \rangle\big)^2
\nabla \| \D \x + \e \|_2^2
}
{2\langle \A \x + \b, \c \rangle \| \D \x + \e \|_2^3 
}
\vspace{0.5mm}
\\
&=&
\displaystyle 
\dfrac{
\| \D \x + \e \|_2^2
2\langle \A \x + \b, \c \rangle \A^\top \c
-
\big(\langle \A \x + \b, \c \rangle\big)^2
2 \D^\top (\D \x + \e)
}
{2\langle \A \x + \b, \c \rangle \| \D \x + \e \|_2^3 
}
\vspace{0.5mm}
\\
&=&
\displaystyle 
\dfrac{
\| \D \x + \e \|_2^2 \A^\top \c
-
\langle \A \x + \b, \c \rangle
\D^\top (\D \x + \e)
}
{ \| \D \x + \e \|_2^3 
}.
\end{array}
\]
\end{proof}

\section{The proof of Proposition \ref{prop:restract_h}}
\label{app:proof_retract_h}
In this appendix, we give the proof of Proposition \ref{prop:restract_h} for the metric retraction over the ellipsoid.
\begin{proof}
To show that \eqref{prop:retract_ellipse} is a retraction, we use the definition of retraction in \textsection\ref{sec:tools:subsec:background}.
Consider a continuous curve $c:\IR\rightarrow\cE_{\A}^{r-1}$, that is smooth for all variables $(\bzeta, \bxi) \in \IRr\times \IRr$, defined as 
$c(t) = \cR_{\bzeta}(t\bxi) = (\bzeta + t\bxi)/ \sqrt{1 + t^2\| \bxi\|_\cE^2}$.
Then $c(0)=\bzeta$ and
\[
c'(0) 
~\coloneqq~
\dfrac{d c(t)}{d t}\bigg|_{t=0}
~=~
\dfrac{\bxi-t\bzeta\|\bxi\|_\cE^2}{(1+t^2\|\bxi\|_\cE^2)\sqrt{1+t^2\|\bxi\|_\cE^2}}\bigg|_{t=0} 
~=~
\bxi.
\]
So $\cR$ is a retraction for $\cE_{\A}^{r-1}$. 
\end{proof}

\section{Derivation of content in Table~\ref{table:manifold}}\label{sec:W:subsect:singleW}
Consider \eqref{W-subproblem} with $n=1$. 
Following the discussion in \textsection\ref{sec:H}, we get grid of the denominator $\langle \W, \W\A_1 \rangle_\textrm{F} $ in by introducing a constrained problem 
$ \argmax_{\W\geq \zeros}
\langle \B_j , \W \rangle_\textrm{F} \,\st\, \langle \W, \W \A_1 \rangle_\textrm{F} = 1
$.
We note that $\A_1$ is a rank-1 symmetric positive semi-definite (PSD) matrix with two eigenvalues: a single positive eigenvalue and $0$ with multiplicity $r-1$.
Moreover, $\A_1$ has its square-root $\A^{1/2}_1$, so 
$
\langle\W^\top\W,\A_1\rangle_\textrm{F}=
\langle\W^\top\W,\A^{1/2}_1\A^{1/2}_1\rangle_\textrm{F}=
\langle\W\A^{1/2}_1,\W\A^{1/2}_1\rangle_\textrm{F}=
\langle\W, \W\rangle_{\A^{1/2}_1}
$ allows us to define the manifold $\cM_1$ as
\begin{equation}\label{eq:matrix_man}
\cM_1
=
\Big\{
\W\in\IRmr ~|~
\langle\W, \W\rangle_{\A^{1/2}_1} - 1 = 0
\Big\}.
\tag{Shell of twisted spectrahedron}
\end{equation}
The term spectrahedron~\cite{ramana1995some} refers to the eigen-spectrum of a matrix behaves like a polyhedron, while the word ``twisted'' refers to the linear transformation $\A_1$.
\\

\noindent\textbf{Tangent space.}
On a single $\cM_j$, let $h(\W) = \langle \W, \W\A_j \rangle_\textrm{F} - 1$ define $\cM_j$, 
its differential is $\textrm{D}h(\Z)[\Z] = 2 \langle \W\A_j, \Z \rangle_\textrm{F}$.
The tangent space of $\cM_j$ at a reference point $\Z \in \cM_j$, denoted as $T_{\Z}\cM_j$, is 
\begin{equation}\label{def:TangentMatrixMan}
T_{\Z}\cM_j 
\coloneqq 
\Big\{ \W \in \IRmr \,\big|\, \langle \Z \A_j, \W \rangle_\textrm{F} = 0 \Big\}
\,=\, \textrm{KerD}h(\W).
\end{equation}

\noindent\textbf{Projection onto tangent space, and Retraction.}
We define the adjoint operator of $\textrm{D}h(\W)[\Z]$ as $\textrm{D}h(\W)^*[{\alpha}]=2{\alpha}\W\A_j$, where $\alpha$ takes the value $\overline{\alpha}=\dfrac{1}{2}\dfrac{\langle\W,\Z\A_j\rangle_\textrm{F}}{\|\Z\A_j\|_\textrm{F}^2} $ obtained by solving the  least squares problem
$\displaystyle \overline{\alpha} = \argmin_{\alpha \in \IR}
\big\| \W-\textrm{D}h(\Z)^*[\alpha] \big\|_\textrm{F}^2$.
Now by $\textrm{D}h(\W)^*[{\alpha}]$, $\overline{\alpha}$, the  definition of projector onto tangent space~\cite[Def3.60]{boumal2023introduction}, and the property of orthogonal projector~\cite[Eq.7.74]{boumal2023introduction}, we have the projector $
\proj_{T_{\Z}\cM_j}(\W) = \W-\textrm{D}h(\Z)^*[\overline{\alpha}]
$ as shown in the following proposition.
\\

\begin{proposition}\label{prop:orth_projMatrixMan}
The map $\proj_{T_{\Z}\cM_j}:\IRmr \rightarrow T_{\Z}\cM_j$ that brings $\W$ 
onto the tangent space at the reference point $\Z$ defined as
\begin{equation}
\proj_{T_{\Z}\cM_j}(\W)
=
\W - \dfrac{\langle \Z\A_j, \W \rangle_\textrm{F}}{\| \Z \A_j\|_\textrm{F}^2} \Z\A_j
\end{equation}
is an orthogonal projector.
\end{proposition}
\begin{proof}
The proof follows the same arguments used in the proof of Proposition~\ref{prop:projTE_h}.
\end{proof}

\begin{proposition}[Retraction]\label{prop:retract_Mj} The map $\cR_{\Z}: T_{\Z}\cM_j\rightarrow \cM_j$ defined as
\begin{equation}\label{def:proj_map_matrix_man}
\cR_{\Z}(\W)
~=~
\dfrac{\W+\Z}{\|\Z+\W\|_{\A_j^{1/2}}}
\overset{T_{\Z}\cM_j}{=}
\dfrac{\Z+\W}{\sqrt{1 + \| \W\|_{\A_j^{1/2}}^2}}
,
\tag{$\cR$}
\end{equation}
is a metric retraction for the manifold $\cM_j$.
\end{proposition}

\begin{proof}
For a continuous curve $c:\IR\rightarrow\cM_j$ that is smooth for all 
$(\W, \Z) \in \IRmr\times \IRmr$, defined as
\[
c(t) 
~=~
\cR_{\Z}(t\W) 
~=~ 
\dfrac{\Z + t\W}{\sqrt{1 + t^2\| \W\|_{\A_j^{1/2}}^2}}.
\]
We have $c(0)=\Z$ and the following so \eqref{def:proj_map_matrix_man} is a retraction for $\cM_j$:
\[
c'(0) 
~\coloneqq~
\dfrac{d c(t)}{d t}\bigg|_{t=0}
~=~
\dfrac{\W-t\| \W\|_{\A_j^{1/2}}^2\Z}{(1+t^2\|\W\|_{\A_j^{1/2}}^2)\sqrt{1+t^2\|\W\|_{\A_j^{1/2}}^2}}\bigg|_{t=0} 
~=~
\W.
\]
\end{proof}

\section{Details of manifold approach for W-subproblem with $n=2$}\label{sec:app:efficient_RgradW:n2}
Consider $n=2$ in \eqref{W-subproblem}.
By the fact that the Cartesian products of manifolds are manifolds, we consider product space $\cM_1 \times \cM_2$.
Define
\begin{equation}\label{eq:matrix_mantot}
\cM^{[2]}
=
\Big\{
\W\in\IRmr \big|
h(\W) = \zeros_2
\Big\}
\text{ where }
h:\IRmr\rightarrow\IR^2 :
\W \mapsto 
\begin{bmatrix}
\langle\W,\W\A_1\rangle_\textrm{F}-1
\\
\langle\W,\W\A_2\rangle_\textrm{F}-1
\end{bmatrix}.
\tag{Spectrahedra}
\end{equation}
We call the manifold $\cM^{[2]}$ ``spectrahedra'' as it is constructed by spectrahedron.
We now compute its tangent space.
Following arguments in \textsection\ref{sec:H} and \textsection\ref{sec:W:subsect:singleW}, we have the following results that we hide the proofs.
The function $h$ in \eqref{eq:matrix_mantot} has the differential $\textrm{D}h(\W)[\Z] = 2\begin{bmatrix}
\langle\W\A_1, \Z\rangle_\textrm{F}
\\
\langle\W\A_2, \Z\rangle_\textrm{F}
\end{bmatrix}$.
The tangent space of $\cM^{[2]}$ is the set 
$
T_{\Z}\cM^{[2]} 
\coloneqq
\Big\{ 
\W \,\big|\, \langle\W\A_1, \Z\rangle_\textrm{F} = \langle\W\A_2, \Z\rangle_\textrm{F} = 0
\Big\},
$
its adjoint is
\[
\textrm{D}h(\W)^*[{\balpha}]
~=~
2\alpha_1\W\A_1 + 2\alpha_2\W\A_2 
~=~
2\W (\alpha_1\A_1 + \alpha_2\A_2).
\]
Let \textrm{vec} be vectorization and 
let $\otimes_{\text{K}}$ be the Kronecker product,
now the $\balpha$ that minimizes 
$ \big\| \W - \textrm{D}h(\Z)^*[{\balpha}] \big\|_\textrm{F}^2$ also minimizes
\[
\begin{array}{lll}
 \Big\| \textrm{vec}\Big( 
\W - 2\Z(\alpha_1\A_1 + \alpha_2 \A_2)
\Big)
\Big\|_2^2
&=&
\displaystyle 
\Big\| \textrm{vec}(\W)  
- 2 (\I \otimes_{\text{K}} \Z ) \textrm{vec}
\Big(
\alpha_1 \A_1 + \alpha_2 \A_2
\Big)
\Big\|_2^2
\\
&=&\bigg\| \textrm{vec}( \W)  
- 2 (\I \otimes_{\text{K}} \Z ) 
\begin{bmatrix}
     \textrm{vec}\A_1 ~ \textrm{vec}\A_2
\end{bmatrix}
\begin{bmatrix}
\alpha_1
\\
\alpha_2
\end{bmatrix}
\bigg\|_2^2.
\end{array}
\]

To simplify the notation, let $\S_{\Z} = (\I \otimes_{\text{K}} \Z ) 
[\textrm{vec}\A_1 ~ \textrm{vec}\A_2]$, where the subscript $\Z$ indicates the dependence of $\Z$.
Now $\balpha$ is the root of the following normal equation 
\[
\S_{\Z}^\top \S_{\Z}
\begin{bmatrix}
    \alpha_1^*
    \\
    \alpha_2^*
\end{bmatrix}
~=~
\dfrac{1}{2}
\S_{\Z}^\top \textrm{vec}\W
~~\implies~~
\balpha^*
~=~
\dfrac{1}{2}(\S_{\Z}^\top\S_{\Z})^{-1}\S_{\Z}^\top
\textrm{vec}\W
~=~
\dfrac{1}{2}\S_{\Z}^\dagger
\textrm{vec}\W.
\]
Now the orthogonal projector $\proj_{T_{\Z}\cM^{[2]}}: \IRmn \mapsto \cM^{[2]}$ is 
\begin{equation}\label{M2:proj}
\proj_{T_{\Z}\cM^{[2]}}(\W)
~=~
\W-2\Z\sum_{j=1}^2
\dfrac{1}{2}\Big(\S_{\Z}^\dagger \textrm{vec}\W\Big)_j\A_j
~=~
\W  - 2 \Z (\alpha_1^* \A_1 + \alpha_2^* \A_2).
\end{equation}
The Riemannian gradient is then
\begin{equation}\label{M2:grad}
\grad F(\W)
~=~
2\proj_{T_{\Z}\cM^{[2]}}\Big( \nabla F(\W) \Big)
~\overset{\eqref{M2:proj}}{=}~
\nabla F(\W) - 2 \Z (\alpha_1^* \A_1 + \alpha_2^* \A_2).    
\end{equation}
Note that the value of $\alpha_1^*$ and $\alpha_2^*$ in
\eqref{M2:grad} 
is an implicit function of $\Z$ and $\nabla F(\W)$ as we have that
$
\begin{bmatrix}
\alpha_1^*
\\
\alpha_2^* 
\end{bmatrix}
=
 \dfrac{1}{2}
\begin{bmatrix}
\Big(
\S_{\Z}^\dagger
\textrm{vec}\nabla F(\W)
\Big)_1 
\\
\Big(
\S_{\Z}^\dagger
\textrm{vec}\nabla F(\W) 
\Big)_2
\end{bmatrix}
$
.
So the explicit expression of the Riemannian gradient is 
\[
\grad F(\W)
~=~
\nabla F(\W) - 
\Z
\Big(
\S_{\Z}^\dagger
\textrm{vec}\nabla F(\W)
\Big)_1  \A_1
-
\Z
\Big(
\S_{\Z}^\dagger
\textrm{vec}\nabla F(\W)
\Big)_2 \A_2
.
\]
Then we compute $\grad^+ F = \max\{\grad F, \zeros\}$ and $\grad^- F = \max\{-\grad F, \zeros\}$ to proceed with RMU.

In conclusion, we can see that to run RMU on $\cM^{[2]}$, there are challenges: 
\begin{itemize}
\item the computation of $\alpha^*$, which includes the computation of $\S_{\Z} \in \IR^{m^2 \times 2}$, 
$\S_{\Z}^{\dagger} \in \IR^{ 2 \times 2}$, in which all these terms have to be re-computed in each iteration.

\item the computation of the metric projection onto $\cM^{[2]}$ itself is a difficult problem.
\end{itemize}

\backmatter

\bmhead{Acknowledgements}
FE is member of the Gruppo Nazionale Calcolo Scientifico - Istituto Nazionale di Alta Matematica (GNCS-INdAM) and she would like to thank Prof. Nicoletta Del Buono from University of Bari Aldo Moro for scientific discussion during this project.

\section*{Declarations}

\begin{itemize}
\item Funding:
F.E. is supported by ERC Seeds Uniba project ``Biomes Data Integration with Low-Rank Models'' (CUP H93C23000720001), Piano Nazionale di Ripresa e Resilienza (PNRR), Missione 4 ``Istruzione e Ricerca''-Componente C2 Investimento 1.1, ``Fondo per il Programma Nazionale di Ricerca e Progetti di Rilevante Interesse Nazionale'', Progetto PRIN-2022 PNRR, P2022BLN38, Computational approaches for the integration of multi-omics data. CUP: H53D23008870001, by INdAM - GNCS Project ``Ottimizzazione e Algebra Lineare Randomizzata'' (CUP E53C24001950001).

\item Code availability: 
MALTAB code is available at \url{https://github.com/flaespo/Chordal_NMF}

\item Author contribution: 
All authors contribute equally to this work.
\end{itemize}

\bibliography{refs}
\end{document}